\documentclass[letterpaper,twoside,noamsfonts,11pt]{amsart}

\pdfoutput=1

\usepackage{michael-org}
\usepackage{textcomp}
\usepackage{verbatim}
\usepackage{float}

\usetikzlibrary{decorations, calc, intersections, through, positioning, math}
\usetikzlibrary{decorations.markings}
\usetikzlibrary{pgfplots.fillbetween}

\DeclareFieldFormat{eprint:ResearchGate}{Available via \printfield{eprinttype}:   \href{https://www.researchgate.net/publication/\strfield{eprint}}{\strfield{eprint}}}

\hypersetup{
}

\allowdisplaybreaks

\pgfplotsset{compat=1.16}

\begin{document}
  \title{Fourier restriction to smooth enough curves}
  \author{Michael Jesurum}
  \address{Department of Mathematics, University of Wisconsin, Madison, WI 53706}
  \email{jesurum@wisc.edu}
  \begin{abstract}
    We prove Fourier restriction estimates to arbitrary compact \(C^{N}\) curves
    for any \(N > d\) in the (sharp) Drury range, using a power of the
    affine arclength measure as a mitigating factor. In particular, we make no
    nondegeneracy assumption on the curve.
  \end{abstract}
  \maketitle
  \section{Introduction} \label{introduction section}
%
The boundedness of restriction operators
\(\mathcal{R} \colon L^{p}(\mathbb{R}^{d})
\rightarrow L^{q}(\gamma; \mathrm{d} \sigma)\) associated with curves
\(\gamma \colon \mathbb{R} \rightarrow \mathbb{R}^{d}\) has been studied for
decades. The first natural choice of measure \(\mathrm{d} \sigma\) is the
Euclidean arclength measure. In this case the maximal range of \(p\) and \(q\)
for which a restriction operator exists depends on the order of vanishing of
the torsion.
\begin{definition}
  For a compact interval \(I \subset \mathbb{R}\) and a curve
  \(\gamma \in C^{d}(I; \mathbb{R}^{d})\), define the \textbf{torsion}
  \(\tau(t) = |\det[\gamma'(t), \dots, \gamma^{(d)}(t)]|\). Furthermore, for
  \(\epsilon \geq 0\) define the weight
  \begin{equation} \label{affine arclength weight}
    w_{\epsilon}(t) = \tau(t)^{\frac{2}{d(d + 1)} + \epsilon}.
  \end{equation}
\end{definition}
In particular, \(\omega_{0}(t)\mathrm{d}t\) is the well-studied affine arclength
measure, which is also interesting due to the affine invariance of the problem.
See~\cite{Guggenheimer} for background on affine geometry. Using the affine
arclength measure, Drury~\cite{Drury} discovered a proof for the optimal range of
\(p\) and \(q\) in the least-degenerate case: when \(\gamma\) is the moment curve
\(\gamma(t) = (t, t^{2}, \dots, t^{d})\). For various classes of curves, many
authors (see below for the history) have shown that the affine arclength
measure compensates for vanishing of the torsion, so restriction
operators exist for nearly the same range of \(p\) and \(q\) as the moment curve.
Several authors have also used the overdamped affine arclength measure,
\(\epsilon > 0\) in~\eqref{affine arclength weight}, to attain the exact range of
\(p\) and \(q\) for the moment curve. The case of a general curve
\(\gamma \in C^{\infty}(I)\) has long
been expected to behave similarly. Building on techniques
in~\cite{ChenFanWang, DendrinosWright, Drury, Stovall}, our main result establishes
the boundedness of restriction operators for arbitrary compact curves that are
smooth enough. In the theorem,
\(C^{N} \coloneqq C^{\lfloor N\rfloor, N-\lfloor N\rfloor}\).
\begin{theorem} \label{restriction theorem}
  Let \(d \geq 2\), \(I \subset \mathbb{R}\) be a compact interval,
  \(N \in \mathbb{R}\) such that \(N > d\), and 
  \(\gamma \in C^{N}(I; \mathbb{R}^{d})\). For \(\epsilon \geq 0\) let
  \(w_{\epsilon}\) be the weight defined in~\eqref{affine arclength weight}.
  Let
  \begin{equation} \label{restriction theorem p and q}
    1 \leq p < \frac{d^{2} + d + 2}{d^{2} + d}
    \quad \text{and} \quad
    \begin{cases}
      1
        \leq q
        \leq \frac{2}{d^{2}+d}p'
        \quad & \text{if } \epsilon > \sum_{j=1}^{d} \frac{1}{N-j},
      \\
      1
        \leq q
        < \frac{\frac{2}{d^{2}+d}+\epsilon}
          {1 + \frac{d^{2}+d}{2}\sum_{j=1}^{d} \frac{1}{N-j}}p'
        \quad & \text{if }
          0 \leq \epsilon \leq \sum_{j=1}^{d} \frac{1}{N-j}.
    \end{cases}
  \end{equation}
  Then there is \(C = C(I, d, \gamma, N, p, q, \epsilon) > 0\) such
  that for any \(f \in L^{p}(\mathbb{R}^{d})\),
  \begin{equation*}
    \bigg(\int_{I} |\hat{f}(\gamma(t))|^{q} w_{\epsilon}(t)
      \mathrm{d} t\bigg)^{\frac{1}{q}}
    \leq C \|f\|_{L^{p}(\mathbb{R}^{d})}.
  \end{equation*}
\end{theorem}
\begin{remark}
  When \(\gamma \in C^{\infty}(I)\), Theorem~\ref{restriction theorem}
  gives a restriction bound for \(q\) on the scaling line
  \(q = \frac{2}{d^{2}+d}p'\) in the overdamped case with any \(\epsilon > 0\)
  and for all \(q < \frac{2}{d^{2}+d}p'\) with the affine arclength measure.
\end{remark}
\begin{remark}
  In the case \(\epsilon \leq \sum_{j = 1}^{d} \frac{1}{N - j}\), the range of
  \(q\) in~\eqref{restriction theorem p and q} is empty whenever
  \begin{equation*}
    1
    < \frac{\frac{2}{d^{2} + d} + \epsilon}
      {1 + \frac{d^{2} + d}{2}\sum_{j = 1}^{d} \frac{1}{N - j}}p'.
  \end{equation*}
  Thus, to obtain restriction estimates for all \(p\) in the range
  \(1 \leq p < \frac{d^{2} + d + 2}{d^{2} + d}\), we need
  \begin{equation} \label{N epsilon interplay for Drury range}
    \sum_{j = 1}^{d} \frac{1}{N - j}
    \leq \frac{4}{(d^{2} + d)^{2}} + \frac{\epsilon(d^{2} + d + 2)}{d^{2} + d}.
  \end{equation}
  If \(N\) is much larger than \(d\),
  then~\eqref{N epsilon interplay for Drury range} is true even in the undamped
  case \(\epsilon = 0\). When \(N\) is closer to \(d\), there is always some
  \(0 < \epsilon_{0} < \sum_{j = 1}^{d} \frac{1}{N - j}\) such
  that~\eqref{N epsilon interplay for Drury range} holds for all
  \(\epsilon \geq \epsilon_{0}\). See Figure~\ref{fig: range of p and q} below,
  which uses the extension operator instead of the restriction operator for
  visual clarity.
\end{remark}
\begin{figure}[H]
  \centering
  \begin{tikzpicture}[xscale = 11.92, yscale = 11.92]
    \draw[->] (0, 0) -- (1.02, 0) node [right] {\(\frac{1}{q'}\)};
    \draw[->] (0, 0) -- (0, .3) node [above] {\(\frac{1}{p'}\)};
    \draw[dashed] 
    (0, {1/4}) node [left] {\(\frac{2}{d^{2}+d+2}\)} -- (1, {1/4});
    \draw[dashed] (1, 0) node [below] {1} -- (1, {1/4});
    \foreach \n in {160, ..., 660}{
      \tikzmath{
        \hue = (660-\n)/20+55;
      }
      \draw[white!\hue!darkgray] (1, 0) -- (0, {\n/2640});
    }
    \foreach \n in {0, ..., 500}{
      \tikzmath{
        \hue = (500-\n)/10+5;
      }
      \draw[white!\hue!darkgray] (1, 0) -- ({\n/2000}, {1/4});
    }
    \fill[white!85!darkgray] (1, 0) -- (0, 0) -- (0, {2/33}) -- cycle;
    \draw[dashed] (1, 0) -- (0, {2/33}) node[left] {\(\epsilon = 0\)};
    \draw[dashed] (1, 0) -- node[below, rotate = {atan(-1/4)}]
      {``='' in \eqref{N epsilon interplay for Drury range}} (0, {1/4});
    \draw (1, 0) -- node[above, rotate = {atan(-1/3)}]
      {\(\epsilon > \sum_{j=1}^{d} \frac{1}{N-j}\)} ({1/4}, {1/4});
  \end{tikzpicture}
  \caption{\small \textit{Range of \(q'\) and \(p'\) for which the
  extension operator associated with an arbitrary \(\gamma \in C^{d+1}\) is
  bounded from \(L^{q'}\) to \(L^{p'}\), with shading indicating the amount of
  damping required.}}
  \label{fig: range of p and q}
\end{figure}
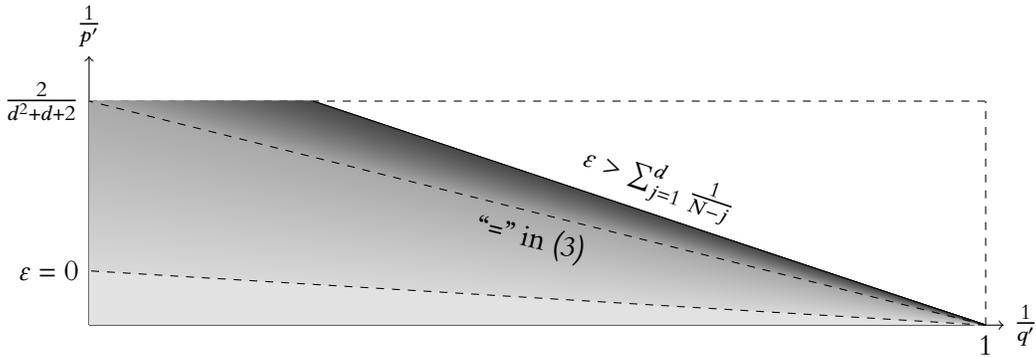
In the case of the moment curve, where the torsion is constant, the offspring
curve method is available to prove optimal restriction estimates. That method
includes an analysis of the function
\begin{equation} \label{sum of gammas}
  \Phi_{\gamma}(t_{1}, \dots, t_{d}) = \gamma(t_{1})+\dots+\gamma(t_{d}).
\end{equation}
When \(\gamma\) is the moment curve, the Jacobian \(J_{\Phi_{\gamma}}\) is (up
to a multiplicative constant) equal to the Vandermonde determinant
\begin{equation} \label{vandermonde}
  v(t_{1}, \dots, t_{d}) = C_{d}\prod_{1 \leq i < j \leq d} (t_{j}-t_{i}).
\end{equation}
To prove Theorem~\ref{restriction theorem}, we apply the Drury method to
intervals on which the torsion of \(\gamma\) is comparable to some dyadic
value. We also need to ensure that \(\Phi_{\gamma}\) is well-behaved on each
interval. The main difficulty is controlling the number of intervals we
study. This is the content of the following theorem:
\begin{theorem} \label{geometric theorem}
  Let \(I \subset \mathbb{R}\) be a compact interval. If
  \(N \in \mathbb{R}\) with \(N > d\) and \(\gamma \in C^{N}(I; \mathbb{R}^{d})\),
  then there is a family of intervals \(I_{k, l}\) such that
  \begin{equation} \label{interval decomposition}
    I
    = \{t \in I : \tau(t) = 0\}
      \cup \bigcup_{k \geq C_{\gamma}} \bigcup_{l = 1}^{N_{k}} I_{k, l},
  \end{equation}
  \begin{equation*} 
    2^{-k-2} \leq \tau(t) \leq 2^{-k+1} \text{ for } t \in I_{k, l},
  \end{equation*}
  \begin{equation} \label{injective}
    (t_{1}, \dots, t_{d}) \mapsto \Phi_{\gamma}(t_{1}, \dots, t_{d})
    \text{ is 1-to-1 for } t_{1} < \dots < t_{d} \in I_{k, l},
  \end{equation}
  \begin{equation} \label{geometric inequality}
    |J_{\Phi_{\gamma}}(t_{1}, \dots, t_{d})|
    \geq C_{d}2^{-k}|v(t_{1}, \dots, t_{d})|
    \text{ for } (t_{1}, \dots, t_{d}) \in (I_{k, l})^{d},
  \end{equation}
  \begin{equation*} 
    N_{k} \leq C_{d, N, \gamma}2^{k\sum_{j=1}^{d}\frac{1}{N-j}}, \text{ and}
  \end{equation*}
  \begin{equation*}
    \sum_{l=1}^{N_{k}}|I_{k, l}|
    \leq C_{d, N, \gamma}(k+C_{d, \gamma})^{d}|I|.
  \end{equation*}
\end{theorem}
\subsection*{History}
%
  %
  %
  In~\cite{Stein1}, Stein traces the roots of Fourier restriction theory
  to observations about the continuity of Fourier transforms of radial
  functions. Laurent Schwartz and many others independently observed that if
  \(1 \leq p < \frac{2d}{d+1}\) and \(f \in L^{p}(\mathbb{R}^{d})\) is a radial
  function, then \(\hat{f}(r)\) is continuous for all \(0 < r < \infty\).
  Thus, \(\hat{f}\) can be thought of as a function on \(S^{d-1}\), even
  though that set has measure 0. Stein wondered if one could give a similar
  statement about Fourier restrictions of nonradial functions. He successfully
  proved such a result in 1967, but did not publish because it was unclear
  what purpose such a lemma would have.

  In 1970, Fefferman~\cite{Fefferman}, in collaboration with Stein, improved
  on Stein's lemma in 2 dimensions and showed that the \(n\)-dimensional
  lemma could be used to make progress on the multiplier problem for the ball.
  Interest in the restriction problem picked up, and by 1974 the case of
  curves in two dimensions had largely been solved by
  Zygmund~\cite{Zygmund}, H{\"ormander}~\cite{Hormander}, and
  Sj{\"o}lin~\cite{Sjolin}. While Zygmund and H{\"ormander} dealt with
  nondegenerate curves (\(\tau \neq 0\)), Sj{\"o}lin brought in the measures
  \(w_{\epsilon}(t)\mathrm{d}t\) to understand curves with vanishing curvature.
  The case of \(d = 2\) and \(\gamma \in C^{\infty}\) in
  Theorem~\ref{restriction theorem} is due to Sj{\"o}lin~\cite{Sjolin}.
  There have been several more
  papers~\cite{Barcelo2, Barcelo1, Fraccaroli, Oberlin2, Oberlin3, Ruiz, Sogge}
  that have answered some remaining questions in 2 dimensions. In the most
  recent of these (in 2021), Fraccaroli~\cite{Fraccaroli} proved a restriction
  theorem in the optimal range for all continuous convex curves, which is
  the first result that did not require \(C^{2}\).

  Unfortunately, the techniques that work well in two dimensions do not
  carry over to higher dimensions. Thus, it was a few more years before any
  results were known. Prestini broke into the high-dimensional setting by
  proving a restriction theorem for curves with nonvanishing torsion in
  1978~\cite{Prestini1} for \(d = 3\) and 1979~\cite{Prestini2} for
  \(d \geq 4\). However, her range of \(p\) was not sharp: it was
  \(1 \leq p < \frac{d^{2}+2d}{d^{2}+2d-2}\) (compare
  with~\eqref{restriction theorem p and q}). She also did not attain bounds
  on the scaling line \(q = \frac{2}{d^{2}+d}p'\), which is seen to be the largest
  possible value of \(q\) by inspecting Knapp examples. In 1982, Christ~\cite{Christ2}
  extended Prestini's theorem to include the scaling line, and then in
  1985~\cite{Christ} he provided restriction estimates for curves of finite type
  with the same range of \(p\) and for \(q\) up to the scaling line. Furthermore,
  those bounds included \(q\) on the scaling line in a restricted range of \(p\).
  The range \(1 \leq p < \frac{d^{2} + 2d}{d^{2} + 2d - 2}\) is called the
  Christ-Prestini range. See also~\cite{Ruiz} for the first result for a
  curve with vanishing torsion.

  At a similar time that Christ was beginning the study of curves with finite
  type, Drury~\cite{Drury} was concluding the study of curves with
  nonvanishing torsion. He proved a restriction theorem for nondegenerate
  \(C^{d}\) curves in the optimal range
  \(1 \leq p < \frac{d^{2} + d + 2}{d^{2} + d}\) and
  \(1 \leq q \leq \frac{2}{d^{2} + d}p'\). Optimality of the range of \(p\) is due to
  Arkhipov, Karacuba, and {\v C}ubarikov~\cite{ArkhipovKaratsubaChubarikov} (see
  also~\cite{ArkhipovKaratsubaChubarikov2}). Further results for nondegenerate
  curves appear in~\cite{BakLee, BakOberlin}.
  
  Shortly thereafter, Drury and Marshall~\cite{DruryMarshall1, DruryMarshall2}
  improved the known estimates for curves of finite type, and then in 1990
  Drury~\cite{Drury2} further improved the results for curves
  \(\gamma(t) = (t, t^{2}, t^{k})\), \(k \geq 4\).

  Little further progress was made for curves with vanishing torsion in higher
  dimensions until 2008, when Bak, Oberlin, and
  Seeger~\cite{BakOberlinSeeger1} solved the monomial curve case. In that and
  a subsequent paper~\cite{BakOberlinSeeger2}, they also obtained endpoint
  results. Shortly thereafter, Dendrinos and M{\"u}ller~\cite{DendrinosMuller}
  obtained results for pertubed monomial curves. General polynomial curves
  were covered for a restricted range of \(p\) by Dendrinos and
  Wright~\cite{DendrinosWright}, and then for the full range of \(p\) by
  Stovall~\cite{Stovall}. See also~\cite{HamLee} for results with general measures.
  
  In addition to monomial curves, Bak, Oberlin, and
  Seeger~\cite{BakOberlinSeeger3} also proved restriction
  theorems for simple curves
  \begin{equation*} 
    \gamma(t) = (t, t^{2}, \dots, t^{d-1}, \phi(t))
  \end{equation*}
  such that \(\phi \in C^{d}\) with \(\phi^{(d)}\) satisfying a
  certain inequality. Chen, Fan, and Wang~\cite{ChenFanWang} were able to
  dispense with the inequality, but at the cost of enforcing
  \(\phi \in C^{N}\) for some \(N > d\). By a change of variables, the
  case of \(d = 2\) and \(\gamma \in C^{N}\) in
  Theorem~\ref{restriction theorem} is due to Chen, Fan, and
  Wang~\cite{ChenFanWang}. Another result
  of this nature appears in~\cite{Wan}.
  
  With the polynomial case solved, it is likely that an
  argument based on \(\epsilon\)-removal and polynomial approximation could be
  used to solve the general \(C^{\infty}\) case off the scaling line. Thus, the
  most interesting new consequences of Theorem~\ref{restriction theorem} are
  (in dimension \(d \geq 3\)) the scaling line estimates for
  \(C^{\infty}(\mathbb{R}^{d})\) curves with \(\epsilon > 0\) and the nontrivial
  range of \(p\) and \(q\) for \(C^{N}(\mathbb{R}^{d})\) curves with
  \(\epsilon = 0\).
\subsection*{Outline of proof}
  Section~\ref{drury section} uses Theorem~\ref{strong geometric theorem}, which is
  a stronger version of Theorem~\ref{geometric theorem}, as a black box to prove
  Theorem~\ref{restriction theorem}. It begins with a restriction
  result on each interval in the decomposition given by
  Theorem~\ref{strong geometric theorem}, and then combines these estimates into a
  restriction inequality on the whole interval.
  Sections~\ref{decomposition section} and~\ref{geometric inequality section}
  are devoted to a proof of Theorem~\ref{strong geometric theorem}.
  Section~\ref{decomposition section} constructs a decomposition
  for Theorem~\ref{strong geometric theorem} in two steps. The first
  step decomposes \(I\) into intervals on which \(\gamma\) is well-behaved, and the secondary decomposition
  creates intervals where certain auxiliary curves are similarly well-behaved.
  Finally, Section~\ref{geometric inequality section} finishes the proof
  of the geometric inequality~\eqref{geometric inequality} and the
  condition~\eqref{injective} on each interval in the decomposition.
\subsection*{Notation}
  Let \(\tilde{I}\) be a compact interval, \(d \in \mathbb{N}\),
  \(N \in \mathbb{R}\) with \(N > d\), and
  \(\gamma \in C^{d}(\tilde{I}; \mathbb{R}^{d})\). These will remain fixed
  throughout this paper, and we will prove Theorem~\ref{restriction theorem}
  and Theorem~\ref{geometric theorem} with these fixed values. \(C\) denotes an
  arbitrary constant that may change line by line and is always allowed to depend
  on the dimension \(d\) and the interval \(\tilde{I}\). Any subscripts indicate
  additional dependence: for instance, \(C_{\gamma}\) is a constant that depends
  only on \(\gamma\), the dimension, and the original interval. For two numbers
  \(A\) and \(B\), write \(A \approx B\) if there exist constants \(C\) and \(C'\)
  such that
  \begin{equation*}
    CB \leq A \leq C'B.
  \end{equation*}
  Once again, subscripts indicate additional dependence. Logarithms are taken in base
  2 purely for the convenience of calculations in Section~\ref{decomposition section}.
\subsection*{Acknowledgements}
  The author would like to thank Betsy Stovall for suggesting this project and
  for advising throughout the process. This material is based upon work
  supported by the National Science Foundation under Grant No. DMS-2037851 and
  DMS-1653264.

  \section{Proof of Theorem~\ref{restriction theorem}} \label{drury section}
%
In this section, we use a strengthening of Theorem~\ref{geometric theorem} to
prove Theorem~\ref{restriction theorem}. The first step is to prove a restriction
estimate on each interval of the decomposition~\eqref{interval decomposition}.
Define the family of offspring curves
\begin{equation*}
  \Upsilon
  = \bigg\{\gamma_{h}(t) = \frac{1}{m} \sum_{j=1}^{m} \gamma(t+h_{j})
    : m \in \mathbb{N}, \ h \in \mathbb{R}^{m},
    \ 0 \leq h_{1} \leq \dots \leq h_{m}\bigg\}.
\end{equation*}
For an interval \(I = [a, b]\), set \(I_{h} = [a-h_{1}, b-h_{m}]\). We will use
induction to show that a restriction bound holds uniformly for
\(\gamma_{h} \in \Upsilon\).
\begin{proposition} \label{Drury argument}
  Let \(I \subseteq \tilde{I}\) be a compact interval and \(k \in \mathbb{Z}\).
  Suppose that for every \(\gamma_{h} \in \Upsilon\),
  \begin{equation} \label{injective Drury}
    (t_{1}, \dots, t_{d}) \mapsto \Phi_{\gamma_{h}}(t_{1}, \dots, t_{d})
    \text{ is 1-to-1 for } t_{1} < \dots < t_{d} \in I_{h} \text{ and}
  \end{equation}
  \begin{equation} \label{geometric inequality Drury}
    |J_{\Phi_{\gamma_{h}}}(t_{1}, \dots, t_{d})|
    \geq C2^{-k}|v(t_{1}, \dots, t_{d})|
    \text{ for } (t_{1}, \dots, t_{d}) \in I_{h}^{d}.
  \end{equation}
  Then for
  \begin{equation*}
    1 \leq p < \frac{d^{2}+d+2}{d^{2}+d}
    \quad \text{and}
    \quad q = \frac{2p'}{d^{2}+d},
  \end{equation*}
  we have the restriction inequality
  \begin{equation} \label{restriction bound Drury}
    \bigg(\int_{I_{h}} |\hat{f}(\gamma_{h}(t))|^{q}
      \mathrm{d} t\bigg)^{\frac{1}{q}}
    \leq 2^{\frac{k}{p'}}C_{p}\|f\|_{L^{p}(\mathbb{R}^{d})}
  \end{equation}
  for all \(f \in L^{p}(\mathbb{R}^{d})\) and all \(\gamma_{h} \in \Upsilon\).
\end{proposition}
\begin{proof}
  We adapt Drury's argument from~\cite{Drury}. By duality, it suffices to study
  the extension operator
  \begin{equation*}
    \mathcal{E}_{h}g(x) = \int_{I_{h}} e^{i\gamma_{h}(t) \cdot x}g(t) \mathrm{d}t.
  \end{equation*}
  We will show that
  \begin{equation*}
    \|\mathcal{E}_{h}g\|_{L^{p'}(\mathbb{R}^{d})}
    \leq 2^{\frac{k}{p'}}C_{p}\|g\|_{L^{q'}(I)},
  \end{equation*}
  for
  \begin{equation*}
    1 \leq q' < \frac{d^{2} + d + 2}{2},
    \quad \frac{d^{2} + d}{2p'} + \frac{1}{q'} = 1.
  \end{equation*}
  The proof is by induction on \(q'\). Hausdorff-Young shows that the base case
  \(q' = 1\) and \(p' = \infty\) is true. The induction hypothesis is that for
  some \(1 \leq q_{0}' < \frac{d^{2}+d+2}{d^{2}+d}\) and \(p_{0}'\) defined by
  \begin{equation*}
    \frac{d^{2}+d}{2p_{0}'}+\frac{1}{q_{0}'} = 1,
  \end{equation*}
  the following inequality holds uniformly for \(\gamma_{h} \in \Upsilon\):
  \begin{equation} \label{induction hypothesis}
    \|\mathcal{E}_{h}g\|_{L^{p_{0}'}(\mathbb{R}^{d})}
    \leq 2^{\frac{k}{p_{0}'}}C_{p_{0}}\|g\|_{L^{q_{0}'}(I)}.
  \end{equation}
  Fix \(\gamma_{\tilde{h}} \in \Upsilon\). For ease of notation, set
  \(\zeta = \gamma_{\tilde{h}}\), \(I = I_{\tilde{h}}\), and
  \(\mathcal{E} = \mathcal{E}_{\tilde{h}}\). To
  improve the bound~\eqref{induction hypothesis}, we first write
  \begin{equation*}
    \Big(\mathcal{E}g\Big(\frac{x}{d}\Big)\Big)^{d}
    = \bigg(\int_{I} e^{i\zeta(t) \cdot \frac{x}{d}}g(t) \mathrm{d}t\bigg)^{d}
    = \int_{I^{d}} e^{ix\cdot\frac{1}{d}\sum_{j=1}^{d} \zeta(t_{j})}
      \prod_{j=1}^{d} g(t_{j}) \mathrm{d}t_{1} \dots \mathrm{d}t_{d}.
  \end{equation*}
  Set
  \begin{equation*}
    A = \{(t_{1}, \dots, t_{d}) \in I^{d} : t_{1} < \dots < t_{d}\}.
  \end{equation*}
  By symmetry in \(t_{1}, \dots, t_{d}\),
  \begin{equation*}
    \Big(\mathcal{E}g\Big(\frac{x}{d}\Big)\Big)^{d}
    = d!\int_{A} e^{ix\cdot\frac{1}{d}\sum_{j=1}^{d} \zeta(t_{j})}
      \prod_{j=1}^{d} g(t_{j}) \mathrm{d}t_{1} \dots \mathrm{d}t_{d}.
  \end{equation*}
  With the change of variables
  \begin{equation*}
    t = t_{1},
    \quad h_{j} = t_{j} - t_{1} \text{ for } 1 \leq j \leq d,
  \end{equation*}
  and with \(B\) the image of \(A\) under this change of variables, we observe that
  \begin{equation*}
    \Big(\mathcal{E}g\Big(\frac{x}{d}\Big)\Big)^{d}
    = d!\int_{B} e^{ix\cdot\frac{1}{d}\sum_{j=1}^{d} \zeta(t+h_{j})}
      \prod_{j=1}^{d} g(t+h_{j})
      \mathrm{d}t \mathrm{d}h_{2} \dots \mathrm{d}h_{d}.
  \end{equation*}
  For fixed \(h_{2}, \dots, h_{d}\), each curve
  \begin{equation*}
    t \mapsto \frac{1}{d}\sum_{j = 1}^{d} \zeta(t + h_{j})
  \end{equation*}
  is an offspring curve in the family \(\Upsilon\). Let \(v\) be the Vandermonde
  determinant~\eqref{vandermonde} and define
  \begin{equation*}
    TG(x)
    = \int_{B} e^{ix\cdot\frac{1}{d}\sum_{j=1}^{d} \zeta(t+h_{j})}
      G(t, h) v(h) \mathrm{d}t \mathrm{d}h.
  \end{equation*}
  \begin{lemma}
    We have the bound
    \begin{equation} \label{induction bound}
      \|TG\|_{L^{p_{0}'}}
      \leq 2^{\frac{k}{p_{0}'}}C_{p_{0}}
        \|G\|_{L^{1}_{h'}(L^{q_{0}'}_{t}; |v(h)|)}.
    \end{equation}
  \end{lemma}
  \begin{proof}
    An application of Minkowski's inequality for integrals shows that 
    \begin{equation*}
      \|TG\|_{L^{p_{0}'}}
      \leq \int \bigg\|\int
        e^{ix\cdot\frac{1}{d}\sum_{j=1}^{d} \zeta(t+h_{j})}
        G(t, h) \mathrm{d} t\bigg\|_{L^{p_{0}'}(\mathrm{d} x)} |v(h)|
        \mathrm{d} h.
    \end{equation*}
    Employing the induction hypothesis~\eqref{induction hypothesis}, we obtain
    \begin{equation*}
       \|TG\|_{L^{p_{0}'}}
      \leq 2^{\frac{k}{p_{0}'}}C_{p_{0}}
        \int \|G(\cdot, h)\|_{L^{q_{0}'}(I_{h})} |v(h)| \mathrm{d} h.
    \end{equation*}
    The lemma now follows from the inequality
    \begin{equation*}
      2^{\frac{k}{p_{0}'}}C_{p_{0}}
        \int \|G(\cdot, h)\|_{L^{q_{0}'}(I_{h})} |v(h)| \mathrm{d} h
      \leq 2^{\frac{k}{p_{0}'}}C_{p_{0}}
        \|G\|_{L^{1}_{h'}(L^{q_{0}'}_{t}; |v(h)|)}.
      \qedhere
    \end{equation*}
  \end{proof}
  \begin{lemma}
    We have the bound
    \begin{equation} \label{L^2 bound}
      \|TG\|_{L^{2}(\mathrm{d} x)}
      \leq 2^{\frac{k}{2}}C\|G\|_{L^{2}_{h'}(L^{2}_{t}; |v(h)|)}.
    \end{equation}
  \end{lemma}
  \begin{proof}
    Set
    \begin{equation*}
      y = \frac{1}{d}\sum_{j=1}^{d} \zeta(t+h_{j}).
    \end{equation*}
    This change of variables is injective because of~\eqref{injective Drury}. 
    The Jacobian is
    \begin{equation*}
      J(t, h)
      = \frac{1}{d^{d}} J_{\Phi_{\zeta}}(t, t+h_{2}, \dots, t+h_{d}).
    \end{equation*}
    The geometric inequality~\eqref{geometric inequality Drury} guarantees that
    \begin{equation} \label{lower bound of jacobian}
      J(t, h) \geq C2^{-k}v(h).
    \end{equation}
    With these variables, set
    \begin{equation*}
      F(y) = \mathbb{1}_{B}(t, h)G(t, h)\frac{v(h)}{J(t, h)}.
    \end{equation*}
    Applying the change of variables to \(T\), we see that
    \begin{equation*}
      TG(x)
      = \int e^{iy \cdot x}F(y) \mathrm{d}y
      = \check{F}(x).
    \end{equation*}
    Plancherel gives \(\|\check{F}\|_{2} = \|F\|_{2}\), so
    \(\|TG\|_{L^{2}(\mathrm{d} x)} = \|F\|_{L^{2}(\mathrm{d}y)}.\)
    Changing variables back and unwinding the definition of \(F\) yields
    \begin{equation} \label{L^2 bound post change of variables}
       \|TG\|_{L^{2}(\mathrm{d} x)}
      = \bigg(\int_{B} |G(t, h)|^{2} \bigg[\frac{v(h)}{J(t, h)}\bigg] v(h)
        \mathrm{d} t \mathrm{d} h\bigg)^{\frac{1}{2}}.
    \end{equation}
    Lines~\eqref{lower bound of jacobian}
    and~\eqref{L^2 bound post change of variables} combine to demonstrate that
    \begin{equation*}
      \|TG\|_{L^{2}(\mathrm{d} x)}
      \leq 2^{\frac{k}{2}}C\bigg(\int_{B} |G(t, h)|^{2}v(h)
        \mathrm{d}t \mathrm{d}h\bigg)^{\frac{1}{2}}.
    \end{equation*}
    Since
    \begin{equation*}
      \bigg(\int_{B} |G(t, h)|^{2} v(h)
        \mathrm{d} t \mathrm{d} h\bigg)^{\frac{1}{2}}
      \leq \|G\|_{L^{2}_{h'}(L^{2}_{t}; |v(h)|)},
    \end{equation*}
    the inequality~\eqref{L^2 bound} is true.
  \end{proof}
  Interpolation of~\eqref{induction bound},~\eqref{L^2 bound}, and the
  trivial \(L^{1}(L^{1}) \rightarrow L^{\infty}\) estimate establishes
  \begin{equation} \label{interpolation bound}
    \|TG\|_{L^{c}}
    \leq 2^{\frac{k}{c'}}C_{a, b}\|G\|_{L^{a}_{h'}(L^{b}_{t}; |v(h)|)}
  \end{equation}
  for all \((a^{-1}, b^{-1})\) in the triangle with vertices \((1, 1)\),
  \(\big(1, \frac{1}{q_{0}'}\big)\), and \(\big(\frac{1}{2}, \frac{1}{2}\big)\),
  with \(c\) satisfying
  \begin{equation*}
    \frac{(d+2)(d-1)}{2}a^{-1}+b^{-1}+\frac{d^{2}+d}{2}c^{-1}
    =\frac{d^{2}+d}{2}.
  \end{equation*}
  In particular, the choice of
  \begin{equation} \label{G definition}
    G(t, h) = |v(h)|^{-1}\prod_{j=1}^{d} g(t+h_{j})
  \end{equation}
  has
  \begin{equation*}
    \|G\|_{L^{a}_{h'}(L^{b}_{t}; |v(h)|)}
    = \bigg(\int_{\mathbb{R}} |v(h)|^{-(a - 1)} \bigg(\int_{\mathbb{R}^{d-1}} 
      |g(t+h_{1}) \cdots g(t+h_{d})|^{b} \mathrm{d}t\bigg)^{\frac{a}{b}}
      \mathrm{d}h'\bigg)^{\frac{1}{a}}.
  \end{equation*}
  As noted in \cite{Drury}, \(v(0, h')^{-1} \in L^{\frac{d}{2}, \infty}_{h'}\),
  so we can apply H\"older's inequality to obtain
  \begin{equation} \label{G L^aL^b Holder}
    \|G\|_{L^{a}_{h'}(L^{b}_{t}; |v(h)|)} \leq \|g\|_{L^{q', 1}_{t}}^{d},
  \end{equation}
  for
  \begin{equation*}
    \begin{cases}
      1 < a < \frac{d+2}{2},
      \\
      a \leq b < \frac{2a}{d+2-da}, \text{ and}
      \\
      \frac{d}{q'}
      = \frac{(d+2)(d-1)}{2}a^{-1}+b^{-1}-\frac{d(d-1)}{2}.
    \end{cases}
  \end{equation*}
  On the other hand, by the definition of \(G\)~\eqref{G definition},
  \begin{equation} \label{TG = Eg^d}
    TG(x) = \frac{1}{d!}\Big(\mathcal{E}g\Big(\frac{x}{d}\Big)\Big)^{d}.
  \end{equation}
  Combining~\eqref{interpolation bound},~\eqref{G L^aL^b Holder}
  and~\eqref{TG = Eg^d}, we see that
  \begin{equation} \label{extension strong q' to p'}
  \|\mathcal{E}g\|_{L^{p'}} \leq C_{p}2^{\frac{k}{p'}}\|g\|_{L^{q', 1}},
  \end{equation}
  for
  \begin{equation} \label{p of a and b}
    \frac{d}{q'}
    = \frac{(d+2)(d-1)}{2}a^{-1}+b^{-1}-\frac{d(d-1)}{2},
  \end{equation}
  where \(p' = \frac{d^{2}+d}{2} q\), and \(a\) and \(b\)
  satisfy~(Figure~\ref{fig: interpolation picture}):
  \begin{equation*}
    \begin{cases}
      \frac{d}{d+2} < a^{-1} < 1,
      \\
      b^{-1} \leq a^{-1},
      \\
      (d+ 2)a^{-1} - 2b^{-1} < d, \text{ and}
      \\
      (q_{0}'-2)a^{-1}+q_{0}'b^{-1} \geq q_{0}'-1.
    \end{cases}
  \end{equation*}
  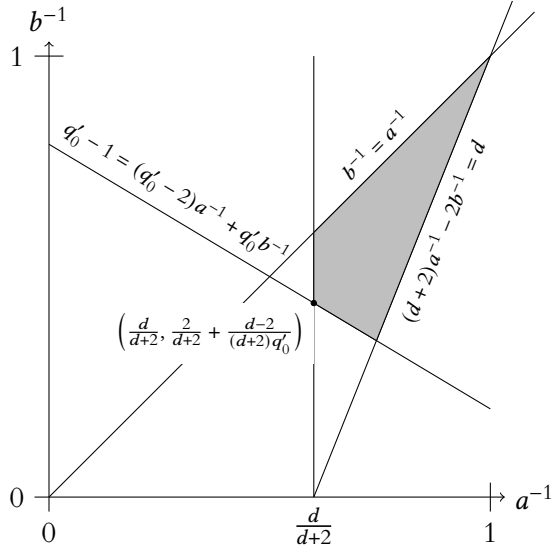
\begin{figure}[H]
    \centering
    \begin{tikzpicture}[x = 2.31 in, y = 2.31 in]
    %
      %
      %
      %
      \draw [->] ($(0, 0) - (0, .5 em)$) node [below] {0}
      -- ($(0, 1) + (0, .5 em)$) node [above] {$b^{-1}$};
      %
      %
      \draw [->] ($(0, 0) - (.5 em, 0)$) node [left] {0}
      -- ($(1, 0) + (.5 em, 0)$) node [right] {$a^{-1}$};
      %
      %
      \draw ($(1, 0) - (0, .5 em)$) node [below] {1}
      -- ($(1, 0) + (0, .5 em)$);
      \draw ($(0, 1) - (.5 em, 0)$) node [left] {1}
      -- ($(0, 1) + (.5 em, 0)$);
      %
      %
      \draw [name path = vertical]
      ({3/5}, 0) node [below] {$\frac{d}{d + 2}$}
      -- ({3/5}, 1);
      %
      %
      \draw [name path = equalslope] (0, 0)
      -- (1.1, 1.1)
      node [above, pos = 0.7, rotate = 45] {\scriptsize $b^{-1} = a^{-1}$};
      %
      %
      \draw [name path = steepslope] ({3/5}, 0)
      -- (1.05, {5/2*1.05 - 3/2})
      node [below, pos = 0.55, rotate = {atan2(5, 2)}]
      {\scriptsize $(d + 2)a^{-1} - 2 b^{-1} = d$};
      %
      %
      \draw [name path = shallowslope] (0, {4/5})
      -- (1, {1/5})
      node [above, pos = 0.26, rotate = {atan2(-3, 5)}]
      {\scriptsize $q_{0}' - 1 = (q_{0}' - 2)a^{-1} + q_{0}'b^{-1}$};
      %
      %
      \path [name intersections = {of = vertical and equalslope}];
      \coordinate (A) at (intersection-1);
      \path [name intersections = {of = equalslope and steepslope}];
      \coordinate (B) at (intersection-1);
      \path [name intersections = {of = steepslope and shallowslope}];
      \coordinate (C) at (intersection-1);
      \path [name intersections = {of = shallowslope and vertical}];
      \coordinate (D) at (intersection-1);
      %
      %
      \filldraw [fill = lightgray] (A) -- (B) -- (C) -- (D) -- cycle;
      %
      %
      \node at (D) [below left, fill = white]
      {\scriptsize
      \(\Big(\frac{d}{d + 2},
        \frac{2}{d + 2} + \frac{d - 2}{(d + 2)q_{0}'}\Big)\)};
      \filldraw (D) circle [radius = 1 pt];
    \end{tikzpicture}
    \caption{\small \textit{Range of \(a\) and \(b\) for
    which~\eqref{extension strong q' to p'} holds with \(q'\)
    satisfying~\eqref{p of a and b} and \(p' = \frac{d^{2} + d}{2} q\).}}
    \label{fig: interpolation picture}
  \end{figure}
  The point
  \((a^{-1}, b^{-1})
  = (\frac{d}{d + 2}, \frac{2}{d + 2} + \frac{d - 2}{(d + 2)q_{0}'})\)
  lies on the boundary of this region and satisfies
  \begin{equation*}
    \frac{(d + 2) (d - 1)}{2}a^{-1} + b^{-1} - \frac{d(d - 1)}{2}
    < \frac{d}{q_{0}'}.
  \end{equation*}
  Taking \((a^{-1}, b^{-1})\) slightly inside of the region and using real
  interpolation, we obtain
  \begin{equation*}
    \|\mathcal{E}g\|_{L^{p'}}
    \leq C_{p}2^{\frac{k}{p'}}\|g\|_{L^{q'}}, 
    \quad p' = \frac{d^{2} + d}{2}q, 
    \quad \frac{d}{q'}
    > \frac{2}{(d + 2)} + \frac{d - 2}{(d + 2)q_{0}'}.
  \end{equation*}
  This closes the induction and proves Proposition~\ref{Drury argument}.
\end{proof}
Now we turn to the deduction of Theorem~\ref{restriction theorem} from the
following stronger version of Theorem~\ref{geometric theorem}. Recall that the
symbol ``\(\approx\)'' depends only on the dimension \(d\) and the original
interval \(\tilde{I}\), unless otherwise specified by subcripts. In particular,
there is no dependence on \(m\) or \(h\) in what follows.
\begin{theorem} \label{strong geometric theorem}
  Let \(\tilde{I} \subset \mathbb{R}\) be a compact interval. If
  \(N \in \mathbb{R}\) with \(N > d\) and
  \(\gamma \in C^{N}(\tilde{I}; \mathbb{R}^{d})\), then there is a
  family of intervals \(I_{k, l}\) such that for every \(m \in \mathbb{N}\)
  and \(h \in \mathbb{R}^{m}\),
  \begin{equation} \label{strong interval decomposition}
    \tilde{I}
    = \{t \in \tilde{I} : \tau(t) = 0\}
      \cup \bigcup_{k \geq C_{\gamma}} \bigcup_{l = 1}^{N_{k}} I_{k, l},
  \end{equation}
  \begin{equation} \label{strong torsion size}
    \tau_{h}(t) \approx 2^{-k} \text{ for } t \in (I_{k, l})_{h},
  \end{equation}
  \begin{equation*} 
    (t_{1}, \dots, t_{d}) \mapsto \Phi_{\gamma_{h}}(t_{1}, \dots, t_{d})
    \text{ is 1-to-1 for } t_{1} < \dots < t_{d} \in (I_{k, l})_{h},
  \end{equation*}
  \begin{equation*} 
    |J_{\Phi_{\gamma_{h}}}(t_{1}, \dots, t_{d})|
    \geq C2^{-k}|v(t_{1}, \dots, t_{d})|
    \text{ for } (t_{1}, \dots, t_{d}) \in (I_{k, l})_{h}^{d},
  \end{equation*}
  \begin{equation} \label{strong number of intervals}
    N_{k} \leq C_{N, \gamma}(k+C_{\gamma})^{d}2^{k\sum_{j=1}^{d}\frac{1}{N-j}},
    \text{ and}
  \end{equation}
  \begin{equation} \label{strong total length of intervals}
    \sum_{l=1}^{N_{k}}|I_{k, l}|
    \leq C_{N, \gamma}(k+C_{\gamma})^{d}.
  \end{equation}
\end{theorem}
\begin{proof} [Deduction of Theorem~\ref{restriction theorem}]
  Let \(\{I_{k, l}\}\) be the intervals given in
  Theorem~\ref{strong geometric theorem}. By~\eqref{strong interval decomposition},
  \begin{equation*}
    \int_{\tilde{I}} |\hat{f}(\gamma(t))|^{q} w_{\epsilon}(t) \mathrm{d} t
    \leq \sum_{k \geq C_{\gamma}} \sum_{l = 1}^{N_{k}}
      \int_{I_{k, l}} |\hat{f}(\gamma(t))|^{q} w_{\epsilon}(t) \mathrm{d} t.
  \end{equation*}
  Utilizing the bound~\eqref{strong torsion size} on the size of the torsion on
  each interval,
  \begin{equation*}
    \sum_{k \geq C_{\gamma}} \sum_{l = 1}^{N_{k}}
      \int_{I_{k, l}} |\hat{f}(\gamma(t))|^{q} w_{\epsilon}(t) \mathrm{d} t
    \leq C\sum_{k \geq C_{\gamma}} \sum_{l = 1}^{N_{k}}
      2^{\frac{-2k}{d^{2} + d} - k\epsilon}
      \int_{I_{k, l}} |\hat{f}(\gamma(t))|^{q} \mathrm{d} t.
  \end{equation*}
  An application of H\"older's inequality shows that for any
  \(q \leq \frac{2}{d^{2} + d}p'\),
  \begin{align*}
    &\sum_{k \geq C_{\gamma}} \sum_{l = 1}^{N_{k}}
      2^{\frac{-2k}{d^{2} + d} - k\epsilon}
      \int_{I_{k, l}} |\hat{f}(\gamma(t))|^{q} \mathrm{d} t
    \\
    &\leq \sum_{k \geq C_{\gamma}} \sum_{l = 1}^{N_{k}}
      2^{\frac{-2k}{d^{2} + d} - k\epsilon}
      |I_{k, l}|^{1 - \frac{(d^{2} + d)q}{2p'}}
      \bigg(\int_{I_{k, l}} |\hat{f}(\gamma(t))|^{\frac{2p'}{d^{2} + d}}
      \mathrm{d} t\bigg)^{\frac{(d^{2} + d)q}{2p'}},
  \end{align*}
  where \(|I_{k, l}|\) is the length of the interval \(I_{k, l}\). Each
  interval \(I_{k, l}\) satisfies the hypotheses of
  Proposition~\ref{Drury argument}, so
  by~\eqref{restriction bound Drury},
  \begin{align*}
    &\sum_{k \geq C_{\gamma}}\sum_{l=1}^{N_{k}}
      2^{\frac{-2k}{d^{2}+d}-k\epsilon}|I_{k, l}|^{1-\frac{(d^{2}+d)q}{2p'}}
      \bigg(\int_{I_{k, l}} |\hat{f}(\gamma(t))|^{\frac{2p'}{d^{2}+d}}
      \mathrm{d}t\bigg)^{\frac{(d^{2}+d)q}{2p'}}
    \\
    &\leq \sum_{k \geq C_{\gamma}}\sum_{l=1}^{N_{k}}
      2^{\frac{-2k}{d^{2}+d}-k\epsilon}|I_{k, l}|^{1-\frac{(d^{2}+d)q}{2p'}}
      2^{\frac{kq}{p'}}C_{p}^{q}\|f\|_{L^{p}(\mathbb{R}^{d})}^{q}.
  \end{align*}
  Another application of H\"older's inequality yields
  \begin{align*}
    &\sum_{k \geq C_{\gamma}} \sum_{l = 1}^{N_{k}}
      2^{\frac{-2k}{d^{2} + d} - k\epsilon}
      |I_{k, l}|^{1 - \frac{(d^{2} + d)q}{2p'}}
      2^{\frac{kq}{p'}}C_{p}^{q}\|f\|_{L^{p}(\mathbb{R}^{d})}^{q}
    \\
    &\leq C_{p}^{q}\|f\|_{L^{p}(\mathbb{R}^{d})}^{q}
      \sum_{k \geq C_{\gamma}}
      2^{\frac{-2k}{d^{2} + d} - k\epsilon + \frac{kq}{p'}}
      N_{k}^{\frac{(d^{2} + d)q}{2p'}}
      \sum_{l = 1}^{N_{k}} |I_{k, l}|.
  \end{align*}
  With the bounds~\eqref{strong number of intervals} on \(N_{k}\)
  and~\eqref{strong total length of intervals} on the total lengths of the
  intervals \(I_{k, l}\),
  \begin{equation*}
    C_{p}^{q}\|f\|_{L^{p}}^{q}
      \sum_{k \geq C_{\gamma}}
      2^{\frac{-2k}{d^{2}+d}-k\epsilon+\frac{kq}{p'}}
      N_{k}^{\frac{(d^{2}+d)q}{2p'}}
    \leq C_{p, q, N, \gamma, \epsilon}\|f\|_{L^{p}}^{q}
      \sum_{k \geq 1}k^{d}
      2^{\frac{-2k}{d^{2}+d}-k\epsilon+\frac{kq}{p'}
        +k\frac{(d^{2}+d)q}{2p'}\sum_{j=1}^{d}\frac{1}{N-j}}.
  \end{equation*}
  The sum converges whenever
  \begin{equation*}
    \frac{(d^{2}+d)q}{2p'}\sum_{j=1}^{d}\frac{1}{N-j}-\frac{2}{d^{2}+d}-\epsilon
      +\frac{q}{p'}
    < 0,
  \end{equation*}
  which occurs in either of the cases:
  \begin{equation*}
    \begin{cases}
      1
        \leq q
        \leq \frac{2}{d^{2} + d}p'
        \quad & \text{if } \epsilon > \sum_{j = 1}^{d} \frac{1}{N - j}
      \\
      1
        \leq q
        < \frac{\frac{2}{d^{2} + d} + \epsilon}
          {1 + \frac{d^{2} + d}{2}\sum_{j = 1}^{d} \frac{1}{N - j}}p'
        \quad & \text{if }
          0 \leq \epsilon \leq \sum_{j = 1}^{d} \frac{1}{N - j}.
    \end{cases}
    \qedhere
  \end{equation*}
\end{proof}
  \section{The Decomposition}
%
This section contains the decomposition for
Theorem~\ref{strong geometric theorem}, of which
Theorem~\ref{geometric theorem} is essentially a special case. First, we will
create an initial decomposition using Lemma 8 from~\cite{ChenFanWang} to find
intervals on which we can prove Theorem~\ref{strong geometric theorem} for the
original curve \(\gamma\). Then, we will use polynomial approximation and
Lemma 2.3 from~\cite{Stovall} to decompose further into intervals on which
offspring curves are well-behaved.

More concretely, the methods in~\cite{DendrinosWright} that we need to prove
Theorem~\ref{strong geometric theorem} on each interval in our final decomposition
require an examination of minors of the torsion matrix for all offspring curves.
With that in mind, for a curve \(\zeta\), a permutation \(\sigma \in S_{d}\)
(the symmetric group on \(d\) elements), and \(1 \leq j \leq d\), define
\begin{equation*} 
  L_{\sigma, j}^{\zeta}(t)
  = \det\begin{pmatrix}
    \zeta_{\sigma(1)}'(t) & \cdots & \zeta_{\sigma(1)}^{(j)}(t) \\
    \vdots                &        & \vdots                     \\
    \zeta_{\sigma(j)}'(t) & \cdots & \zeta_{\sigma(j)}^{(j)}(t)
  \end{pmatrix}.
\end{equation*}
Whenever \(j = d\), we will omit \(\sigma\) since \(|L_{\sigma, d}^{\zeta}|\) does
not depend on \(\sigma\). We also omit \(\sigma\) when \(\sigma\) is the identity.
Recall that \(\gamma \in C^{N}\) for some \(d < N \in \mathbb{R}\). The main
result of this section is the following proposition.
\begin{proposition} \label{full decomposition proposition}
  For every \(k_{d} \in \mathbb{Z}\), there is a family of intervals
  \(\{I_{l}\}\) and permutations \(\sigma_{l}\) such that for every
  \(m \in \mathbb{N}\) and \(h \in \mathbb{R}^{m}\),
  \begin{equation*}
    \{t \in \tilde{I} : 2^{-k_{d}-1} \leq |L_{d}^{\gamma}(t)| \leq 2^{-k_{d}}\}
      \subseteq \bigcup_{l} I_{l},
  \end{equation*}
  \begin{equation*} 
    |L_{j, \sigma_{l}}^{\gamma_{h}}(t)| \approx 2^{-k_{j}},
    \quad \forall t \in (I_{l})_{h},\ 1 \leq j \leq d, \text{ and}
  \end{equation*}
  \begin{equation} \label{full decomposition number of intervals}
    \#\{I_{l}\}
    \leq C_{N, \gamma}(k_{d}+C_{\gamma})^{d}2^{k_{d}\sum_{j=1}^{d}\frac{1}{N-j}}.
  \end{equation}
  \begin{equation*} 
    \sum_{l} |I_{k, l}| \leq C(k_{d}+C_{\gamma})^{d}.
  \end{equation*}
\end{proposition}
\subsection*{The initial decomposition}
  We first prove Proposition~\ref{full decomposition proposition} in the
  special case \(h = 0\).
  \begin{proposition} \label{k decomposition lemma}
    For every \(k_{d} \in \mathbb{Z}\), there is a family of intervals \(\{I_{l}\}\)
    with
    \begin{equation}  \label{initial decomposition number and length of intervals}
      \sum_{I \in \mathcal{I}_{k_{d}}} |I| \leq C(k_{d}+C_{\gamma})^{d}
      \quad \text{and} \quad
      \#\{I_{l}\} \leq C_{N, \gamma}2^{k_{d}\sum_{j=1}^{d}\frac{1}{N-j}}
    \end{equation}
    such that
    \begin{equation*}
      \{t \in \tilde{I} : 2^{-k_{d} - 1} \leq |L_{d}^{\gamma}(t)| \leq 2^{-k_{d}}\}
      \subseteq \bigcup_{I \in \mathcal{I}_{k_{d}}} I.
    \end{equation*}
    Furthermore, there are constants \(A_{j}\) depending only on \(\gamma\), \(j\),
    and \(d\) such that on each interval \(I \in \mathcal{I}_{k_{d}}\), there is a
    permutation \(\sigma \in S_{d}\) and \(k_{j} \in \mathbb{Z}\) with
    \(A_{j} \leq k_{j} \leq k_{d} + A_{j} + \log(d!\|\gamma\|_{C^{d}}^{d})\)
    such that
    \begin{equation} \label{size of L_j}
      2^{-k_{j} - 2} \leq |L_{\sigma, j}^{\gamma}(t)| \leq 2^{-k_{j} + 1},
      \quad t \in I,\ 1 \leq j \leq d.
    \end{equation}
  \end{proposition}
  The first step in proving Proposition~\ref{k decomposition lemma} is to show
  that for each \(t \in I\), there is a permutation \(\sigma\) such that the
  \(L_{\sigma, j}^{\gamma}(t)\)'s are generally decreasing in \(j\).
  \begin{lemma} \label{number of k_j}
    There are constants \(A_{j}\) such that for every \(t \in I\), there is a
    permutation \(\sigma\) such that if
    \begin{equation*}
      2^{-k_{d}-1} \leq |L_{d}^{\gamma}(t)| \leq 2^{-k_{d}},
    \end{equation*}
    then there is \(k_{j} \in \mathbb{Z}\) with
    \(A_{j} \leq k_{j} \leq k_{d} + A_{j} + \log(d!\|\gamma\|_{C^{d}}^{d})\)
    such that
    \begin{equation*}
      2^{-k_{j}-1} \leq |L_{\sigma, j}^{\gamma}(t)| \leq 2^{-k_{j}}.
    \end{equation*}
  \end{lemma}
  \begin{proof}
    Fix \(t \in I\). For each \(j\), let \(k_{j}\) be the unique integer that
    satisfies
    \begin{equation*}
      2^{-k_{j}-1} \leq |L_{\sigma, j}^{\gamma}(t)| < 2^{-k_{j}}.
    \end{equation*}
    For any permutation \(\sigma\),
    \begin{equation*}
      L_{\sigma, j}^{\gamma} \leq j!\|\gamma\|_{C^{d}}^{j}.
    \end{equation*}
    Hence,
    \begin{equation} \label{L_j's are close lower bound}
      k_{j} \geq -\log(j!\|\gamma\|_{C^{d}}^{j}).
    \end{equation}
    To get an upper bound on \(k_{j}\), we'll show by induction that there is a
    permutation \(\sigma\) such that
    \begin{equation} \label{L_j's are close upper bound}
      |L_{\sigma, j}^{\gamma}(t)|
      \geq \frac{j!}{d!\|\gamma\|_{C^{d}}^{d - j}}|L_{d}^{\gamma}(t)|
	\quad \text{ for all } 1 \leq j \leq d.
    \end{equation}
    The base case of \(j = d\) is true for every permutation. In the induction
    step, suppose that \(\sigma(d), \dots, \sigma(j+2)\) have been specified
    (with nothing specified if \(j = d-1\)). Let \(M\) be the torsion matrix
    with the last \(d-j-1\) columns deleted and rows
    \(\sigma(d), \dots, \sigma(j+2)\) deleted. Then
    \(|\det M| = |L_{j+1}^{\sigma}|\). For \(1 \leq i, l \leq j\), let
    \(M_{i, l}\) be the minor of \(M\) obtained by deleting the \(i\)'th row
    and \(l\)'th column. Using the cofactor expansion, we find that
    \begin{equation*}
      |L_{j+1}^{\sigma}(t)|
      = \bigg|\sum_{i = 1}^{j+1}
	(-1)^{i+j+1}\gamma_{i}^{(d)}(t)\det M_{i, d}\bigg|
      \leq \|\gamma\|_{C^{d}} \sum_{i = 1}^{j+1} |\det M_{i, j+1}|.
    \end{equation*}
    Hence there is some \(i\) such that
    \begin{equation*}
      |\det M_{i, j+1}|
      \geq \frac{1}{(j+1)\|\gamma\|_{C^{d}}}|L_{j+1}^{\sigma}(t)|.
    \end{equation*}
    Thus, for every permutation \(\sigma\) that sends \(i\) to \(j+1\),
    \begin{equation*}
      |L_{\sigma, j}^{\gamma}(t)|
      \geq \frac{1}{(j+1)\|\gamma\|_{C^{d}}}|L_{j+1}^{\sigma}(t)|.
    \end{equation*}
    By the induction hypothesis,
    \begin{equation*}
      |L_{\sigma, j}^{\gamma}(t)|
      \geq \frac{1}{(j+1)\|\gamma\|_{C^{d}}}
	\frac{(j+1)!}{d!\|\gamma\|_{C^{d}}^{d-j-1}}|L_{d}^{\gamma}(t)|
      = \frac{j!}{d!\|\gamma\|_{C^{d}}^{d-j}}|L_{d}^{\gamma}(t)|.
    \end{equation*}
    This completes the induction, so~\eqref{L_j's are close upper bound} holds.
    Putting together~\eqref{L_j's are close lower bound}
    and~\eqref{L_j's are close upper bound},
    \begin{equation*} 
      -\log(j!\|\gamma\|_{C^{d}}^{j})
      \leq k_{j}
      \leq -\log\Big(\frac{j!}{d!\|\gamma\|_{C^{d}}^{d - j}}\Big) + k_{d}.
    \end{equation*}
    The lemma follows by setting \(A_{j} = -\log(j!\|\gamma\|_{C^{d}}^{j})\).
  \end{proof}
  We can combine Lemma~\ref{number of k_j} and the following lemma
  from~\cite{ChenFanWang} \(d\) times to prove
  Proposition~\ref{k decomposition lemma}.
  \begin{lemma} [\cite{ChenFanWang} Lemma 8] \label{ChenFanWang lemma 8}
    Let \(\varphi \in C^{1 / \alpha}(I)\) with \(\alpha > 0\). For every
    \(k \in \mathbb{Z}\), there exist disjoint intervals
    \(\{I_{k, j} \subseteq I\}_{j = 1}^{N_{k}}\) such that
    \(2^{-k - 2} \leq |\varphi(t)| \leq 2^{-k + 1}\) for all \(t \in I_{k, j}\) and
    \begin{equation*}
      \{t \in I : 2^{-k - 1} \leq |\varphi(t)| \leq 2^{-k}\}
      \subseteq \bigcup_{j = 1}^{N_{k}} I_{k, j};
    \end{equation*}
    moreover, there is a constant \(B_{\alpha}\) such that
    \(N_{k} \leq B_{\alpha} 2^{\alpha k}\) for every \(k\).
  \end{lemma}
  \begin{remark}
    The proof of the above lemma in~\cite{ChenFanWang} shows that we can take
    \begin{equation*}
      B_{\alpha}
      = \|\varphi\|_{C^{\frac{1}{\alpha}}}^{\alpha}4^{\frac{1}{\alpha}+\alpha+4}.
    \end{equation*}
  \end{remark}
  \begin{proof}[Proof of Proposition~\ref{k decomposition lemma}]
    For an interval \(J\), a permutation \(\sigma\), \(1 \leq j \leq d\), and
    \(k \in \mathbb{Z}\), let \(\mathcal{I}_{j}^{\sigma}(J, k)\) be the set of
    intervals from Lemma~\ref{ChenFanWang lemma 8} with
    \(\varphi = L_{\sigma, j}^{\gamma}\)
    and \(\frac{1}{\alpha} = N - j\). To simplify the notation, set
    \begin{equation*}
      B_{j}
      = \|\gamma\|_{C^{N}}^{\frac{1}{N-j}}4^{\frac{1}{N-j}+N-j+4}.
    \end{equation*}
    Then for any interval \(J\), the number of intervals of
    \(\mathcal{I}_{j}^{\sigma}(J, k)\) is at most \(B_{j}2^{\frac{k}{N - j}}\).
    Combining Lemma~\ref{number of k_j} with \(d\)-many applications of
    Lemma~\ref{ChenFanWang lemma 8}, we see that
    \begin{align*}
      &\{t \in \tilde{I} : 2^{-k_{d} - 1} \leq |L_{d}^{\gamma}(t)| \leq 2^{-k_{d}}\}
      \\
      &\subseteq \bigcup_{\sigma \in S_{d}}
	\bigcup_{I_{d} \in \mathcal{I}_{d}^{\sigma}(\tilde{I}, k_{d})} \hspace{-1em}
	\bigcup_{k_{d-1}=A_{d-1}}
	      ^{A_{d-1}+k_{d}+\log(d!\|\gamma\|_{C^{d}}^{d})} \hspace{-1.2em}
	\cdots 
	\bigcup_{I_{2} \in \mathcal{I}_{2}^{\sigma}(I_{3}, k_{2})} \hspace{-1em}
	\bigcup_{k_{1}=A_{1}}^{A_{1}+k_{d}+\log(d!\|\gamma\|_{C^{d}}^{d})}
	  \hspace{-1em}
	\bigcup_{I_{1} \in \mathcal{I}_{1}^{\sigma}(I_{2}, k_{1})} \hspace{-1em}
	I_{1}.
    \end{align*}
    Whenever \(j < d\), the total number of intervals \(I_{j}\) in each pair of
    unions
    \begin{equation} \label{pair of unions}
      \bigcup_{k_{j} = A_{j}}^{A_{j}+k_{d}+\log(d!\|\gamma\|_{C^{d}}^{d})}
	\hspace{-1em}
      \bigcup_{I_{j} \in \mathcal{I}_{j}^{\sigma}(I_{j+1}, k_{j})} \hspace{-1.25em}
	I_{j},
    \end{equation}
    is bounded by
    \begin{equation*}
      \sum_{k_{j} = A_{j}}^{k_{d}+A_{j}+\log(d!\|\gamma\|_{C^{d}}^{d})}
	B_{j}2^{\frac{k_{j}}{N-j}}
      \leq B_{j}2^{\frac{A_{j}}{N-j}}
	\frac{2^\frac{k_{d}+\log(d!\|\gamma\|_{C^{d}}^{d})+1}{N-j}}
	{2^{\frac{1}{N-j}}-1}.
    \end{equation*}
    Recalling that all logarithms are taken in base 2 and using the fact that 
    \(N-j > 1\),
    \begin{equation*}
      B_{j}2^{\frac{A_{j}}{N-j}}
	\frac{2^\frac{k_{d}+\log(d!\|\gamma\|_{C^{d}}^{d})+1}{N-j}}
	{2^{\frac{1}{N-j}}-1}
      \leq C\|\gamma\|_{C^{N}}^{\frac{1}{N-j}}4^{N}
	\bigg(\frac{d!\|\gamma\|_{C^{d}}^{d}}
	{j!\|\gamma\|_{C^{d}}^{j}}\bigg)^{\frac{1}{N-j}}
	2^{\frac{k_{d}}{N-j}}
      \leq C\|\gamma\|_{C^{N}}^{\frac{1}{N-j}}4^{N}
	\|\gamma\|_{C^{d}}^{\frac{d-j}{N-j}}
	2^{\frac{k_{d}}{N-j}}.
    \end{equation*}
    Hence, the set
    \(\{t \in I : 2^{-k_{d} - 1} \leq |L_{d}^{\gamma}(t)| \leq 2^{-k_{d}}\}\) is
    covered by at most
    \begin{equation*}
      C4^{Nd+\frac{1}{N-d}}\|\gamma\|_{C^{N}}^{\sum_{j=1}^{d}\frac{1}{N-j}}
	\|\gamma\|_{C^{d}}^{\sum_{j=1}^{d-1}\frac{d-j}{N-j}}
	2^{k_{d}\sum_{j=1}^{d}\frac{1}{N-j}}
      = C_{N, \gamma}2^{k_{d}\sum_{j=1}^{d}\frac{1}{N-j}}
    \end{equation*}
    many intervals that satisfy~\eqref{size of L_j}. Moreover, the sum of the lengths
    of the intervals in each pair of unions~\eqref{pair of unions} is bounded by
    \begin{equation*}
      (k_{d}+\log(d!\|\gamma\|_{C^{d}}^{d}))|I_{j+1}|,
    \end{equation*}
    since the intervals in \(\mathcal{I}_{j}^{\sigma}(I_{j+1}, k_{j})\) are disjoint
    for every \(k_{j}\). Therefore, the total length of all the intervals in the
    initial decomposition at scale \(k_{d}\) is at most
    \begin{equation*}
      d!(k_{d}+\log(d!\|\gamma\|_{C^{d}}^{d}))^{d}|\tilde{I}|
      = C(k_{d}+C_{\gamma})^{d}.
      \qedhere
    \end{equation*}
  \end{proof}
\subsection*{The secondary decomposition}
  We now proceed to the general \(h \in \mathbb{R}^{m}\) in
  Proposition~\ref{full decomposition proposition}. Our initial decomposition gave
  a family of intervals where \(|L_{j, \sigma}^{\gamma}| \approx 2^{-k_{j}}\). We
  finish the proof of Proposition~\ref{full decomposition proposition} by applying
  the following proposition to each \(\zeta_{j} = (\gamma_{1}, \dots, \gamma_{j})\)
  in turn. We need to ensure the intervals in the initial decomposition are small
  for this proposition. By the upper
  bounds~\eqref{initial decomposition number and length of intervals} on the total
  length and the number of intervals, we can freely shrink the intervals to be of
  size at most \(C_{N, \gamma}2^{-k_{d}\sum_{j=1}^{d}\frac{1}{N-j}}\) while
  retaining the necessary upper
  bound~\eqref{full decomposition number of intervals} on the total number of
  intervals.
  \begin{proposition}
    Let \(I = [a, b]\) with \(b-a \leq 1\) and let
    \(\zeta \in C^{N}(I; \mathbb{R}^{j})\). There is a
    constant \(A\) depending only on \(N\) and \(j\) such that if
    \(|L_{j}^{\zeta}(t)| \approx 2^{-k}\) on \(I\) and
    \begin{equation} \label{size of intervals}
      b-a
      \leq A^{\frac{1}{N-j}}\|\zeta\|_{C^{N}}^{\frac{-j}{N-j}}
        2^{\frac{-k}{N-j}},
    \end{equation}
    there is a decomposition \(I = \cup_{i=1}^{C_{N}} I_{i}\) into disjoint
    intervals such that for every \(m \in \mathbb{N}\) and \(h \in \mathbb{R}^{m}\),
    \begin{equation*}
      |L_{j}^{\zeta_{h}}(t)| \approx 2^{-k}
      \quad \forall t \in (I_{i})_{h}.
    \end{equation*}
  \end{proposition}
  \begin{proof}
  By Taylor's theorem, for any \(t \in I\),
  \begin{align*}
    \zeta(t)
    &= \sum_{i=0}^{\lfloor N \rfloor-1} \frac{\zeta^{(i)}(a)}{i!}(t-a)^{i}
      +\frac{\zeta^{(\lfloor N \rfloor)}(z_{t})}{\lfloor N \rfloor!}
      (t-a)^{\lfloor N \rfloor}
    \\
    &= \sum_{i=0}^{\lfloor N \rfloor} \frac{\zeta^{(i)}(a)}{i!}(t-a)^{i}
      +\frac{\zeta^{(\lfloor N \rfloor)}(z_{t})
      -\zeta^{(\lfloor N \rfloor)}(a)}{\lfloor N \rfloor!}(t-a)^{\lfloor N \rfloor}.
  \end{align*}
  Set
  \begin{equation*}
    P(t) = \sum_{i=0}^{\lfloor N \rfloor} \frac{\zeta^{(i)}(a)}{i!}(t-a)^{i}
    \quad \text{and} \quad
    R(t) = \frac{\zeta^{(\lfloor N \rfloor)}(z_{t})
      -\zeta^{(\lfloor N \rfloor)}(t)}{\lfloor N \rfloor!}(t-a)^{\lfloor N \rfloor}.
  \end{equation*}
  Then \(R \in C^{N}(I)\), \(R^{(i)}(a) = 0\) for
  \(0 \leq i \leq \lfloor N \rfloor\), and
  \begin{equation} \label{remainder bound}
    |R^{(i)}(t)| \leq |t-a|^{N-i}\|\zeta\|_{C^{N}}
    \quad \forall t \in I.
  \end{equation}
  For any \(h \in \mathbb{R}^{m}\) and \(t \in I_{h}\),
  \begin{equation} \label{L_h,j}
    \begin{aligned}
      L_{j}^{\zeta_{h}}(t)
      &= \det[P_{h}'(t)+R_{h}'(t), \dots, P_{h}^{(j)}(t)+R_{h}^{(j)}(t)]
      \\
      &= \det[P_{h}'(t), \dots, P_{h}^{(j)}(t)]
      \\
      &\quad +\det[R_{h}'(t),\ P_{h}''(t), \dots, P_{h}^{(j)}(t)]
      \\
      &\quad
        +\det[P_{h}'(t)+R_{h}'(t),\ R_{h}''(t),\ P_{h}'''(t), \dots, P_{h}^{(j)}(t)]
      \\
      &\quad \dots
      \\
      &\quad +\det[P_{h}'(t)+R_{h}'(t), \dots, P_{h}^{(j-1)}(t)+R_{h}^{(j-1)}(t),\
        R_{h}^{(j)}(t)]
      \\
      &\eqqcolon L^{P_{h}}_{j}(t)+L_{P, R, j, h}(t).
    \end{aligned}
  \end{equation}
  We will show that \(|L_{P, R, j, h}| \lesssim 2^{-k}\) with an implicit constant
  depending on \(A\) in~\eqref{size of intervals} (which we can make as small
  as we need), and then we will divide \(I\) into intervals where
  \(|L_{j}^{P_{h}}| \approx 2^{-k}\). Combined, these imply that
  \(|L_{j}^{\zeta_{h}}| \approx 2^{-k}\). We start with \(|L_{P, R, j, h}|\).
  Applying~\eqref{remainder bound}, we see that
  \begin{equation*}
    |L_{P, R, j, h}(t)| \leq C|I|^{N-j}\|\zeta\|_{C^{N}}^{j}
    \quad \forall t \in I_{h}.
  \end{equation*}
  Let \(A\) be a small constant to be chosen shortly. Assuming the
  bound~\eqref{size of intervals} on the size of each interval holds, we
  conclude that
  \begin{equation} \label{L_P,R,j,h}
    |L_{P, R, j, h}(t)| \leq AC2^{-k}
    \quad \forall t \in I_{h}.
  \end{equation}
  The next lemma will give a decomposition to deal with \(|L_{j}^{P_{h}}|\).
  \begin{lemma} [\cite{Stovall}, Lemma 2.3] \label{Stovall lemma 2.3}
    Fix \(j \geq 2\) and \(l \in \mathbb{N}\). There exists a decomposition of
    \([-1, 1]\) into disjoint intervals \(I_{i}\), \(1 \leq i \leq C_{l}\),
    such that for every \(I_{i}\) and every degree \(l\) polynomial
    \(Q \colon \mathbb{R} \rightarrow \mathbb{R}^{j}\) satisfying
    \begin{equation*}
      |L^{Q}_{j}(t)| \approx 1, \quad t \in [-1, 1],
    \end{equation*}
    every offspring curve \(Q_{h}(t)\) satisfies
    \begin{equation*}
      |L^{Q_{h}}_{j}(t)| \approx_{l, j} 1
      \quad \forall t \in (I_{i})_{h}.
    \end{equation*}
  \end{lemma}
  For \(t \in [-1, 1]\), set
  \begin{equation*}
    Q(t)
    = 2^{\frac{k}{j}}\Big(\frac{2}{b-a}\Big)^{\frac{j+1}{2}}
      \bigg(P_{1}\Big(\frac{b-a}{2}(t+1)+a\Big), \dots,
      P_{j}\Big(\frac{b-a}{2}(t+1)+a\Big)\bigg)
  \end{equation*}
  Before we can apply Lemma~\ref{Stovall lemma 2.3}, we calculate
  \begin{equation*}
    |L^{Q}_{j}(t)|
    = \Big|2^{k}\Big(\frac{2}{b-a}\Big)^{\frac{j(j+1)}{2}}
      \Big(\frac{b-a}{2}\Big)^{\frac{j(j+1)}{2}}
      L^{P}_{j}\Big(\frac{b-a}{2}(t+1)+a\Big)\Big|
    \quad \forall t \in [-1, 1].
  \end{equation*}
  By the calculation in~\eqref{L_h,j} with \(h = 0\), we have
  \begin{equation*}
    L_{j}^{\zeta}(t) = L_{j}^{P}(t)+L_{P, R, j, 0}(t).
  \end{equation*}
  Since \(|L_{j}^{\zeta}| \approx 2^{-k}\) and \(|L_{P, R, j, 0}(t)| \leq AC2^{-k}\)
  by~\eqref{L_P,R,j,h}, we can choose \(A\) small enough that
  \(|L^{P}_{j}| \approx 2^{-k}\). Therefore,
  \begin{equation*}
    |L^{Q}_{j}(t)| \approx 1
    \quad \forall t \in [-1, 1].
  \end{equation*}
  Applying Lemma~\ref{Stovall lemma 2.3}, we obtain a decomposition
  \([-1, 1] = \cup_{i=1}^{C_{N}} J_{i}\) into disjoint
  intervals such that
  \begin{equation*}
    |L^{Q_{h}}_{j}(t)| \approx 1
    \quad \forall t \in J_{i},
  \end{equation*}
  Setting
  \begin{equation*}
    I_{i} = \Big\{t \in I : \frac{b-a}{2}(t-1)+a \in J_{i}\Big\},
  \end{equation*}
  we see that
  \begin{equation} \label{L^P_h_j}
    |L^{P_{h}}_{j}(t)| \approx 2^{-k}
    \quad \forall t \in I_{i}.
  \end{equation}
  As before, we can choose \(A\) small enough that we can
  combine~\eqref{L_h,j},~\eqref{L_P,R,j,h}, and~\eqref{L^P_h_j} to conclude
  \(|L_{j}^{\zeta_{h}}(t)| \approx 2^{-k}\).
  \end{proof}
 \label{decomposition section}
  \section{The Geometric Inequality}
%
To finish the proof of Theorem~\ref{strong geometric theorem}, we need to show that
\(\Phi_{\gamma_{h}}\) given in~\eqref{sum of gammas} is 1-to-1 for
\(t_{1} < \dots < t_{d}\) and that the geometric
inequality~\eqref{geometric inequality} holds on each interval given
in Lemma~\ref{k decomposition lemma}. The following proposition shows that the
geometric inequality~\eqref{geometric inequality} holds on each interval in our
decomposition for \(\gamma\), all offspring curves \(\gamma_{h}\), and all
truncations of these curves.
\begin{proposition} \label{geometric inequality truncation proposition}
  Let \(I \subset \mathbb{R}\), \(n \in \mathbb{N}\),
  \(\zeta \colon I \rightarrow \mathbb{R}^{n}\), and
  \begin{equation*}
    \Phi_{\zeta}(t_{1} \cdots t_{n})
    = \zeta(t_{1}) + \cdots + \zeta(t_{n}).
  \end{equation*}
  Assume that there are \(k_{j} \in \mathbb{Z}\), \(1 \leq j \leq n\), such that
  \begin{equation} \label{size of L_j^zeta}
    |L_{j}^{\zeta}| \approx 2^{-k_{j}} \quad \text{on } I.
  \end{equation}
  Then the Jacobian
  \begin{equation*}
    J_{\Phi_{\zeta}}(t_{1}, \dots, t_{n}) = \det[\zeta'(t_{1}) \dots \zeta'(t_{n})]
  \end{equation*}
  satisfies
  \begin{equation*} 
    |J_{\Phi_{\zeta}}(t_{1}, \dots, t_{n})|
    \approx_{n} 2^{-k_{n}}|v(t_{1}, \dots, t_{n})|
  \end{equation*}
  for all \((t_{1}, \dots, t_{n}) \in I^{n}\).
\end{proposition}
The above proposition shows that if \(I\) is an interval in the decomposition,
\(1 \leq n \leq d\), and \(\zeta = ((\gamma_{h})_{1}, \dots, (\gamma_{h})_{n})\),
then the Jacobian \(J_{\Phi_{\zeta_{h}}}\) is single-signed and nonzero in the
region \(A = \{(t_{1}, \dots, t_{n}) \in I^{d} : t_{1} < \dots < t_{n}\}\).
With that, an argument of Steinig~\cite{Steinig}
(see also~\cite{CarberyVanceWaingerWatsonWright, DendrinosWright}) shows that
\(\Phi_{\gamma_{h}}\) is 1-to-1 on \(A\).
\begin{proposition} [Steinig]
  \(\Phi_{\gamma_{h}}\) is 1-to-1 on
  \(A = \{(t_{1}, \dots, t_{d}) \in I^{d} : t_{1} < \dots < t_{d}\}\).
\end{proposition}
For the convenience of the reader, we recall Steinig's argument.
\begin{proof}
  Assume for contradiction that there are \(\vec{s} \neq \vec{t} \in A\) such
  that
  \begin{equation} \label{Steinig contradiction}
    \gamma_{h}(s_{1}) + \dots + \gamma_{h}(s_{d})
    \neq \gamma_{h}(t_{1}) + \dots + \gamma_{h}(t_{d}).
  \end{equation}
  We can rewrite~\eqref{Steinig contradiction} as
  \begin{equation*}
    \sum_{j = 1}^{m} \epsilon_{j}\gamma_{h}(u_{j}) = 0
  \end{equation*}
  for some even integer \(m \in [2, 2d]\), \(u_{1} < \dots < u_{m} \in I\),
  \(\epsilon_{j} \in \{-1, 1\}\), and \(\sum_{j = 1}^{m} \epsilon_{j} = 0\).
  Let
  \begin{equation*}
    \alpha_{l} = \sum_{j=1}^{l} \epsilon_{j}, \quad 1 \leq l \leq m.
  \end{equation*}
  Then the sequence of \(\alpha_{l}\)'s has at most \(d-1\) changes of sign.
  Define the step function \(\phi(u)\) to be \(\alpha_{j}\) when
  \(u \in (u_{j}, u_{j+1})\). We have
  \begin{equation} \label{linear dependence}
    0
    = \sum_{j = 1}^{m} \epsilon_{j}\gamma_{h}(u_{j})
    = \sum_{j = 1}^{m-1} \alpha_{j}[\gamma_{h}(u_{j}) - \gamma_{h}(u_{j+1})]
    = -\int_{u_{1}}^{u_{m}} \phi(u)\gamma_{h}'(u) \mathrm{d} u.
  \end{equation}
  Let \(I_{i}\), \(1 \leq i \leq n\) be the ordered, maximal intervals where
  \(\phi\) is constant and nonzero. Since the sequence of \(\alpha_{l}\)'s has
  at most \(d-1\) changes of sign, \(n \leq d\). Let \(M\) be
  the \(n \times n\) matrix whose \((i, j)\)'th entry is given by
  \begin{equation*}
    \int_{I_{i}} |\phi(u)|(\gamma_{h})_{j}'(u) \mathrm{d} u.
  \end{equation*}
  Setting \(\zeta = ((\gamma_{h})_{1}, \dots, (\gamma_{h})_{n})\),
  \begin{equation*}
    \det M
    = \itint{I_{1}}{\ }{I_{n}}{} |\phi(u_{1})| \dots |\phi(u_{n})|
      \det[\zeta'(u_{1}) \cdots \zeta'(u_{n})]
      \mathrm{d} u_{1} \dots \mathrm{d} u_{n}.
  \end{equation*}
  By~\eqref{linear dependence}, the rows of \(M\) are linearly dependent, so
  \(\det M = 0\). On the other hand, \(J_{\Phi_{\zeta}}\) is single-signed and
  nonzero. Thus
  \begin{equation*}
    0
    = \det M
    = \itint{I_{1}}{\ }{I_{n}}{} |\phi(u_{1})| \dots |\phi(u_{n})|
      \det[\zeta'(u_{1}) \cdots \zeta'(u_{n})]
      \mathrm{d} u_{1} \dots \mathrm{d} u_{n}
    > 0,
  \end{equation*}
  so we have reached a contradiction.
\end{proof}
\subsection*{Proof of Proposition~\ref{geometric inequality truncation proposition}}
  The proof comes in two steps. Both parts of this proof are adaptations of methods
  in~\cite{DendrinosWright}. Some minor differences arise because we are not
  dealing with polynomials.
  
  First, we will define a sequence of iterated integrals
  \(\mathfrak{I}_{1}, \dots, \mathfrak{I}_{n}\) such that
  \begin{equation} \label{jacobian equals iterated integral}
    J_{\Phi_{\zeta}}(t_{1}, \dots, t_{n})
    = \mathfrak{I}_{n}(t_{1}, \dots, t_{n}).
  \end{equation}
  The equality~\eqref{jacobian equals iterated integral} will be shown in
  Lemma~\ref{jacobian equals iterated integral lemma}. Then, using the inductive
  definition of the iterated integrals, we will show in
  Lemma~\ref{size of iterated integrals lemma} that
  \begin{equation*}
    |\mathfrak{I}_{n}(t_{1}, \dots, t_{n})|
    \approx_{n} 2^{-k_{n}}v(t_{1}, \dots, t_{n}).
  \end{equation*}
  To that end, let
  \begin{equation*} 
    \mathfrak{I}_{1}(t_{1})
    = \frac{L_{n-2}(t_{1})L_{n}(t_{1})}{[L_{n-1}(t_{1})]^{2}}.
  \end{equation*}
  For \(2 \leq m \leq n\) and \(t_{1}, \dots, t_{m} \in I^{m}\), define
  \begin{equation} \label{I_m}
    \mathfrak{I}_{m}(t_{1}, \dots, t_{m})
    = \bigg(\prod_{j=1}^{m}
      \frac{L_{n-m-1}(t_{j}) L_{n-m+1}(t_{j})}{[L_{n-m}(t_{j})]^{2}}\bigg)
      \itint{t_{1}}{t_{2}}{t_{m-1}}{t_{m}}
        \mathfrak{I}_{m-1}(\vec{s})\mathrm{d}s_{1} \dots \mathrm{d}s_{m-1},
  \end{equation}
  with the convention that \(L_{0} = L_{-1} \equiv 1\). As mentioned, the proof of
  Proposition~\ref{geometric inequality truncation proposition} will be complete
  following the proofs of Lemmata~\ref{jacobian equals iterated integral lemma}
  and~\ref{size of iterated integrals lemma}.
  \begin{lemma} \label{jacobian equals iterated integral lemma}
    \(\mathfrak{I}_{n}\) defined in~\eqref{I_m}
    satisfies~\eqref{jacobian equals iterated integral}.
  \end{lemma}
  \begin{proof}
    Define \(f_{i, 0} = \zeta_{i}\) for \(1 \leq i \leq n\), and for
    \(1 \leq j \leq i-1\) define
    \begin{equation*}
      f_{i, j} = \frac{f_{i, j-1}'}{f_{j, j-1}'}.
    \end{equation*}
    Assume for now that the denominator is always nonzero; this follows
    from~\eqref{geometric inequality lemma 4.1 application} and the
    condition~\eqref{size of L_j^zeta}. We will show that
    \begin{equation} \label{iterated integrals equal f's}
      \mathfrak{I}_{n-j+1}(t_{1}, \dots, t_{n-j+1})
      = \det\begin{pmatrix}
          f_{j, j-1}'(t_{1}) & \dots & f_{j, j-1}'(t_{n-j+1}) \\
          \vdots       &       & \vdots       \\
          f_{n, j-1}'(t_{1}) & \dots & f_{n, j-1}'(t_{n-j+1})
        \end{pmatrix}
    \end{equation}
    for all \(1 \leq j \leq n\). In particular, when \(j = 1\) we see that
    \begin{equation*}
      \mathfrak{I}_{n}(t_{1}, \dots, t_{n})
      = \det\begin{pmatrix}
          f_{1, 0}'(t_{1}) & \dots & f_{1, 0}'(t_{n}) \\
          \vdots       &       & \vdots       \\
          f_{n, 0}'(t_{1}) & \dots & f_{n, 0}'(t_{n})
        \end{pmatrix}
      = \det\begin{pmatrix}
          \zeta_{1}'(t_{1}) & \dots & \zeta_{1}'(t_{n}) \\
          \vdots       &       & \vdots       \\
          \zeta_{n}'(t_{1}) & \dots & \zeta_{n}'(t_{n})
        \end{pmatrix}
      = J_{\Phi_{\zeta}}(t_{1}, \dots, t_{n}).
    \end{equation*}
    The proof of~\eqref{iterated integrals equal f's} requires two ingredients.
    First, we need to write down the exact relationship between each \(f_{i, j}'\)
    and various derivatives of \(\zeta\). Second, we need an iterative way of writing
    the left-hand side of~\eqref{iterated integrals equal f's}.
    
    For the first ingredient, we will need to define auxiliary matrices
    \begin{equation*}
      L_{\zeta_{i_{1}}, \dots, \zeta_{i_{l}}}(t)
      = \det\begin{pmatrix}
        \zeta_{i_{1}}'(t) & \cdots & \zeta_{i_{l}}^{(l)}(t) \\
        \vdots            &        & \vdots                 \\
        \zeta_{i_{1}}'(t) & \cdots & \zeta_{i_{l}}^{(l)}(t)
      \end{pmatrix}.
    \end{equation*}
    If \(A\) is the \((j+1) \times (j+1)\) matrix defining
    \(L_{\zeta_{1} \dots \zeta_{j} \zeta_{i}}\), and if
    \([r_{1}, \dots, r_{j}; c_{1}, \dots, c_{j}]\) denotes the determinant of
    the matrix obtained from \(A\) by deleting rows \(r_{1}, \dots, r_{j}\) and
    columns \(c_{1}, \dots, c_{j}\), then an application of Sylvester's Determinant
    Identity (see~\cite{Bareiss}) gives
    \begin{equation*}
      [j, j+1; j, j+1] \cdot \det A
      = [j+1; j+1] \cdot [j; j]-[j+1; j] \cdot [j; j+1].
    \end{equation*}
    Unwinding all the definitions, we see that
    \begin{align*}
      [j, j + 1; j, j + 1] &= L_{\zeta_{1} \dots \zeta_{j - 1}},
      \\
      [j + 1; j + 1] &= L_{\zeta_{1} \dots \zeta_{j}},
      \\
      [j; j] &= (L_{\zeta_{1} \dots \zeta_{j - 1} \zeta_{i}})',
      \\
      [j + 1; j] &= (L_{\zeta_{1} \dots \zeta_{j}})', \text{ and}
      \\
      [j; j + 1] &= L_{\zeta_{1} \dots \zeta_{j - 1} \zeta_{i}}.
    \end{align*}
    Thus, we have
    \begin{equation*}
      L_{\zeta_{1} \dots \zeta_{j - 1}}
        \cdot L_{\zeta_{1} \dots \zeta_{j} \zeta_{i}}
      = L_{\zeta_{1} \dots \zeta_{j}}
        \cdot (L_{\zeta_{1} \dots \zeta_{j - 1} \zeta_{i}})'
        - (L_{\zeta_{1} \dots \zeta_{j}})'
        \cdot L_{\zeta_{1} \dots \zeta_{j - 1} \zeta_{i}}.
    \end{equation*}
    Since \(L_{\zeta_{1} \dots \zeta_{j}} = L_{j}^{\zeta}\) is bounded away from 0
    by~\eqref{size of L_j^zeta}, the above shows that
    \begin{equation*}
      \bigg(\frac{L_{\zeta_{1} \dots \zeta_{j - 1} \zeta_{i}}}
        {L_{\zeta_{1} \dots \zeta_{j}}}\bigg)'
      = \frac{L_{\zeta_{1} \dots \zeta_{j - 1}}
        L_{\zeta_{1} \dots \zeta_{j} \zeta_{i}}}
        {(L_{\zeta_{1} \dots \zeta_{j}})^{2}}.
    \end{equation*}
    Induction in \(i\) and \(j\) then gives
    \begin{equation} \label{geometric inequality lemma 4.1 application}
      f_{i, j}'
      = \bigg(\frac{f_{i, j-1}'}{f_{j, j-1}'}\bigg)'
      = \bigg(\frac{L_{\zeta_{1} \dots \zeta_{j - 1} \zeta_{i}}}
        {L_{\zeta_{1} \dots \zeta_{j}}}\bigg)'
      = \frac{L_{\zeta_{1} \dots \zeta_{j - 1}}
        L_{\zeta_{1} \dots \zeta_{j} \zeta_{i}}}
        {(L_{\zeta_{1} \dots \zeta_{j}})^{2}}.
    \end{equation}
    The second ingredient is covered by the following calculus lemma
    in~\cite{DendrinosWright}:
    \begin{lemma} [\cite{DendrinosWright} Lemma 5.1]
       Let \(\{g_{i}\}_{i = 1}^{l}\) be smooth functions on an open interval
       \(J \subset \mathbb{R}\) such that \(g_{1}\) never vanishes on \(J\).
       If \(f_{i} = \frac{g_{i}}{g_{1}}\), \(2 \leq i \leq l\), then for
      \((t_{1}, \dots, t_{l}) \in J^{l}\),
      \begin{equation*}
        \setlength\arraycolsep{2pt}
        \det\begin{pmatrix}
          g_{1}(t_{1}) & \dots & g_{1}(t_{l}) \\
          \vdots       &       & \vdots       \\
          g_{n}(t_{1}) & \dots & g_{n}(t_{l})
        \end{pmatrix}
        = \prod_{i = 1}^{l} g_{1}(t_{i})
          \int_{t_{1}}^{t_{2}} \dots \int_{t_{l-1}}^{t_{l}}
          \det\begin{pmatrix}
            f_{2}'(s_{1}) & \dots & f_{2}'(s_{l-1}) \\
            \vdots        &       & \vdots            \\
            f_{n}'(s_{1}) & \dots & f_{n}'(s_{l-1})
          \end{pmatrix}
          \mathrm{d} s_{1} \dots \mathrm{d} s_{l-1}.
      \end{equation*}
    \end{lemma}
    Using this lemma (noting that \(f_{j, j-1}' \neq 0\)), we have
    \begin{equation} \label{geometric inequality lemma 5.1 application}
      \begin{aligned}
        &\det\begin{pmatrix}
          f_{j, j-1}'(t_{1}) & \dots & f_{j, j-1}'(t_{n-j+1}) \\
            \vdots       &       & \vdots       \\
            f_{n, j-1}'(t_{1}) & \dots & f_{n, j-1}'(t_{n-j+1})
        \end{pmatrix}
        \\
        &= \prod_{i=1}^{n-j+1} f_{j, j-1}'(t_{i})
          \int_{t_{1}}^{t_{2}} \dots \int_{t_{n-j}}^{t_{n-j+1}}
          \det\begin{pmatrix}
            f_{j+1, j}'(s_{1}) & \dots & f_{j+1, j}'(s_{n-j}) \\
            \vdots        &       & \vdots            \\
            f_{n, j}'(s_{1}) & \dots & f_{n, j}'(s_{n-j})
          \end{pmatrix}
          \mathrm{d}s_{1} \dots \mathrm{d}s_{n-j}.
      \end{aligned}
    \end{equation}
    As in~\cite{DendrinosWright},
    combining~\eqref{geometric inequality lemma 4.1 application}
    and~\eqref{geometric inequality lemma 5.1 application} iteratively gives us the
    equality \eqref{iterated integrals equal f's}, thus proving the lemma.
  \end{proof}
  \begin{lemma} \label{size of iterated integrals lemma}
    Under the assumption~\eqref{size of L_j^zeta}, for \(1 \leq m \leq n\) we have
    \begin{equation} \label{size of iterated integrals}
      \mathfrak{I}_{m}(t_{1}, \dots, t_{m})
      \approx_{m} \pm2^{(m+1)k_{n-m}-mk_{n-m-1}-k_{n}}v(t_{1}, \dots, t_{m}),
    \end{equation}
    where \(k_{0} = k_{-1} = 0\). In particular,
    \begin{equation*}
    |\mathfrak{I}_{n}(t_{1}, \dots, t_{n})|
    \approx_{n} |2^{-k_{n}}v(t_{1}, \dots, t_{n})|.
    \end{equation*}
  \end{lemma}
  \begin{proof}
    We proceed by induction. In the base case \(m = 1\), the Vandermonde determinant
    \(v(t)\) is simply the constant function 1. Furthermore,
    \begin{equation*}
      \mathfrak{I}_{1}(t_{1})
      = \frac{L_{n-2}^{\zeta}(t_{1})L_{n}^{\zeta}(t_{1})}
        {[L_{n-1}^{\zeta}(t_{1})]^{2}}.
    \end{equation*}
    For every \(t_{1}\), the assumption~\eqref{size of L_j^zeta} shows
    \begin{equation*}
      |\mathfrak{I}_{1}(t_{1})|
      \approx 2^{-k_{n-2}-k_{n}+2k_{n-1}},
    \end{equation*}
    so the base case is complete. In the inductive step,
    assume that~\eqref{size of iterated integrals} holds for some
    \(1 \leq m-1 \leq n\). From the definition of \(\mathfrak{I}_{m}\)
    in~\eqref{I_m} and the condition~\eqref{size of L_j^zeta},
    \begin{align*}
      \mathfrak{I}_{m}(t_{1}, \dots, t_{m})
      &\approx_{m} \pm\bigg(\prod_{j=1}^{m} 2^{2k_{n-m}-k_{n-m-1}-k_{n-m+1}}\bigg)
        \itint{t_{1}}{t_{2}}{t_{m-1}}{t_{m}}
        \mathfrak{I}_{m-1}(s_{1}, \dots, s_{m-1})
        \mathrm{d}s_{1} \dots \mathrm{d}s_{m-1}
      \\
      &= \pm2^{2mk_{n-m}-mk_{n-m-1}-mk_{n-m+1}}\itint{t_{1}}{t_{2}}{t_{m-1}}{t_{m}}
        \mathfrak{I}_{m-1}(s_{1}, \dots, s_{m-1})
        \mathrm{d}s_{1} \dots \mathrm{d}s_{m-1}.
    \end{align*}
    By the induction hypothesis,
    \begin{equation*}
      \mathfrak{I}_{m}(t_{1}, \dots, t_{m})
      \approx_{m} \pm2^{(m+1)k_{n-m}-mk_{n-m-1}-k_{n}}
        \itint{t_{1}}{t_{2}}{t_{m-1}}{t_{m}}
        v(s_{1}, \dots, s_{m-1}) \mathrm{d}s_{1} \dots \mathrm{d}s_{m-1}.
    \end{equation*}
    The integrand is a homogeneous polynomial of degree
    \(\frac{(m-1)(m-2)}{2}\). Thus, the integral is a homogeneous polynomial
    of degree \(\frac{(m-1)(m-2)}{2}+m-1 = \frac{m(m - 1)}{2}\). Hence, there is
    some polynomial \(P\) such that
    \begin{align*}
      \itint{t_{1}}{t_{2}}{t_{m-1}}{t_{m}}
        v(s_{1}, \dots, s_{m-1}) \mathrm{d}s_{1} \dots \mathrm{d}s_{m-1}
      &= P(t_{1}, \dots, t_{m}) \prod_{1 \leq i < j \leq m} (t_{j} - t_{i})
      \\
      &= P(t_{1}, \dots, t_{m}) v(t_{1}, \dots, t_{m}).
    \end{align*}
    Moreover, for any \(1 \leq i < j \leq m\), the integral is 0 whenever
    \(t_{j} = t_{i}\). Since \(v(t_{1}, \dots, t_{m})\) already has degree
    \(\frac{m(m-1)}{2}\), \(P\) must be a constant, so
    \begin{equation*}
      \mathfrak{I}_{m}(t_{1}, \dots, t_{m})
      \approx_{m} \pm2^{(m+1)k_{n-m}-mk_{n-m-1}-k_{n}}v(t_{1}, \dots, t_{m}).
    \end{equation*}
    This closes the induction and finishes the proof of the lemma.
  \end{proof}
%
 \label{geometric inequality section}
  \printbibliography

@article{
  ArkhipovKaratsubaChubarikov,
  author = "G.I. Arkhipov and A. A. Karatsuba and V.N. Chubarikov",
  title = "The index of convergence of the singular integral in Tarry's problem",
  journal = "Dokl. Akad. Nauk SSSR",
  date = "1979",
  volume = "248",
  number = "2",
  pages = "268-272",
  url = "http://mi.mathnet.ru/eng/dan42980",
}

@book{
  ArkhipovKaratsubaChubarikov2,
  author = "G.I. Arkhipov and A. A. Karatsuba and V.N. Chubarikov",
  title = "Trigonometric Sums in Number Theory and Analysis",
  series = "de Gruyter Expositions in Mathematics",
  date = "2004",
  volume = "39",
  isbn = "3-11-016266-0",
  addendum = "Translated from the 1987 Russian original",
}

@article{
  BakLee,
  author = "Jong-Guk Bak and Sanghyuk Lee",
  title = "Estimates for an oscillatory integral operator related to restriction to space curves",
  journal = "Proc. Amer. Math. Soc.",
  volume = "132",
  number = "5",
  pages = "1393-1401",
  year = "2004",
  doi = "10.1090/S0002-9939-03-07144-2",
}

@incollection{
  BakOberlin,
  author = "Jong-Guk Bak and Daniel M. Oberlin",
  title = "A note on {F}ourier restriction for curves in \(\mathbb{R}^{3}\)",
  booktitle = "Harmonic analysis at Mount Holyoke (South Hadley, MA, 2001)",
  series = "Contemp. Math.",
  volume = "320",
  pages = "9-13",
  publisher = "Amer. Math. Soc.",
  year = "2003",
}

@article{
  BakOberlinSeeger1,
  author = "Jong-Guk Bak and Daniel Oberlin and Andreas Seeger",
  title = "Restriction of Fourier transforms to curves and related oscillatory integrals",
  journal = "Amer. J. Math.",
  volume = "131",
  number = "2",
  pages = "277-311",
  year = "2009",
  doi = "10.1353/ajm.0.0044",
}

@article{
  BakOberlinSeeger2,
  author = "Jong-Guk Bak and Daniel Oberlin and Andreas Seeger",
  title = "Restriction of Fourier transforms to curves: An endpoint estimate with affine arclength measure",
  journal = "J. Reine Angew. Math.",
  volume = "682",
  pages = "167-205",
  year = "2013",
  doi = "10.1515/crelle-2012-0042",
}

@article{
  BakOberlinSeeger3,
  author = "Jong-Guk Bak and Daniel Oberlin and Andreas Seeger",
  title = "Restriction of Fourier transforms to curves II: some classes with vanishing torsion",
  journal = "J. Aust. Math. Soc.",
  volume = "85",
  number = "1",
  pages = "1-28",
  year = "2008",
  doi = "10.1017/S1446788708000578",
}

@article{
  Barcelo1,
  author = "Bartolomé Barceló",
  title = "The restriction of the Fourier transform to some curves and surfaces",
  journal = "Studia Mathematica",
  volume = "84",
  number = "1",
  pages = "39-69",
  year = "1986",
  doi = "10.4064/sm-84-1-39-69",
}

@article{
  Barcelo2,
  author = "Bartolomé Barceló",
  title = "On the restriction of the Fourier transform and Fourier series to circles of lacunary radii",
  journal = "Rendiconti del Circolo Matematico di Palermo",
  volume = "35",
  pages = "330-348",
  year = "1986",
  doi = "10.1007/BF02843902",
}

@article{
  Bareiss,
  author = "Erwin H. Bareiss",
  title = "Sylvester’s identity and multistep integer-preserving Gaussian elimination",
  journal = "Math. Comp.",
  volume = "22",
  pages = "565-578",
  year = "1968",
  doi = "10.1090/S0025-5718-1968-0226829-0",
}

@article{
  CarberyVanceWaingerWatsonWright,
  author = "Anthony Carbery and James Vance and Stephen Wainger and David Watson and James Wright",
  title = "\(L^{p}\) estimates for operators associated to flat curves without the Fourier transform",
  journal = "Pacific J. Math.",
  volume = "167",
  number = "2",
  pages = "243-262",
  year = "1995",
  doi = "10.2140/pjm.1995.167.243",
}

@article{
  ChenFanWang,
  author = "Xianghong Chen and Dashan Fan and Lifeng Wang",
  title = "Restriction of the Fourier transform to some oscillating curves",
  journal = "J Fourier Anal Appl",
  volume = "24",
  pages = "1141-1159",
  year = "2018",
  doi = "10.1007/s00041-017-9554-6",
}

@article{
  Christ,
  author = "Michael Christ",
  title = "On the restriction of the Fourier transform to curves: endpoint results and the degenerate case",
  journal = "Trans. Amer. Math. Soc.",
  volume = "287",
  number = "1",
  pages = "223-238",
  year = "1985",
  doi = "10.1090/S0002-9947-1985-0766216-6",
}

@thesis{
  Christ2,
  author = "Michael Christ",
  title = "Restriction of the Fourier transform to submanifolds of low codimension",
  school = "The University of Chicago",
  year = "1982",
  url = "https://pi.lib.uchicago.edu/1001/cat/bib/495239",
}

@article{
  DendrinosMuller,
  author = "Spyridon Dendrinos and Detlef M{\"u}ller",
  title = "Uniform estimates for the local restriction of the Fourier transform to curves",
  journal = "Trans. Amer. Math. Soc.",
  volume = "365",
  number = "7",
  pages = "3477-3492",
  year = "2013",
  doi = "10.1090/S0002-9947-2012-05769-2",
}

@article{
  DendrinosWright,
  author = "Spyridon Dendrinos and James Wright",
  title = "Fourier restriction to polynomial curves I: a geometric inequality",
  journal = "Amer. J. Math",
  volume = "132",
  number = "4",
  pages = "1031-1076",
  year = "2010",
  doi = "10.1353/ajm.0.0127",
}

@article{
  Drury,
  author = "Stephen W. Drury",
  title = "Restrictions of Fourier transforms to curves",
  journal = "Annales de L'Institut Fourier",
  volume = "35",
  number = "1",
  pages = "117-123",
  year = "1985",
  doi = "10.5802/aif.1001",
}

@article{
  Drury2,
  author = "Stephen W. Drury",
  title = "Degenerate curves and harmonic analysis",
  journal = "Math. Proc. Camb. Phil. Soc.",
  volume = "108",
  number = "1",
  pages = "89-96",
  year = "1990",
  doi = "10.1017/S0305004100068973",
}

@article{
  DruryMarshall1,
  author = "S. W. Drury and B. P. Marshall",
  title = "Fourier restriction theorems for curves with affine and Euclidean arclengths",
  journal = "Math. Proc. Camb. Phil. Soc.",
  volume = "97",
  number = "1",
  pages = "111-125",
  year = "1985",
  doi = "10.1017/S0305004100062654",
}

@article{
  DruryMarshall2,
  author = "S. W. Drury and B. P. Marshall",
  title = "Fourier restriction theorems for degenerate curves",
  journal = "Math. Proc. Camb. Phil. Soc.",
  volume = "101",
  number = "3",
  pages = "541-553",
  year = "1987",
  doi = "10.1017/S0305004100066901",
}

@article{
  Fefferman,
  author = "Charles Fefferman",
  title = "Inequalities for strongly singular convolution operators",
  journal = "Acta Mathematica",
  volume = "124",
  pages = "9-36",
  year = "1970",
  doi = "10.1007/BF02394567",
}

@online{
  Fraccaroli,
  author = "Marco Fraccaroli",
  title = "Uniform Fourier restriction for convex curves",
  year = "2021",
  eprint = "2111.06874",
  archivePrefix = "arXiv",
  urldate = "2022-01-31",
}

@book{
  Guggenheimer,
  author = "Heinrich W. Guggenheimer",
  title = "Differential Geometry",
  series = "Dover Books on Mathematics",
  publisher = "Dover Publications",
  year = "2012",
  isbn = "9780486157207",
}

@article{
  HamLee,
  author = "Seheon Ham and Sanghyuk Lee",
  title = "Restriction estimates for space curves with respect to general
  measures",
  journal = "Advances in Mathematics",
  volume = "254",
  pages = "251-279",
  year = "2014",
  doi = "10.1016/j.aim.2013.12.017",
}

@article{
  Hormander,
  author = "Lars H{\"o}rmander",
  title = "Oscillatory integrals and multipliers on \(FL^{p}\)",
  journal = "Arkiv f{\"o}r Matematik",
  volume = "11",
  pages = "1-11",
  year = "1973",
  doi = "10.1007/BF02388505",
}

@article{
  Oberlin,
  author = "Daniel M. Oberlin",
  title = "Convolution estimates for some measures on curves",
  journal = "Proc. Amer. Math. Soc.",
  volume = "99",
  number = "1",
  pages = "56-60",
  year = "1987",
  doi = "10.1090/S0002-9939-1987-0866429-6",
}

@article{
  Oberlin2,
  author = "Daniel M. Oberlin",
  title = "Fourier restriction for affine arclength measures in the plane",
  journal = "Proc. Amer. Math. Soc.",
  volume = "129",
  number = "11",
  pages = "3303-3305",
  year = "2001",
  doi = "10.1090/S0002-9939-01-06012-9",
}

@article{
  Oberlin3,
  author = "Daniel M. Oberlin",
  title = "Two estimates for curves in the plane",
  journal = "Proc. Amer. Math. Soc.",
  volume = "132",
  number = "11",
  pages = "3195-3201",
  year = "2004",
  doi = "10.1090/S0002-9939(04)07610-5",
}

@article{
  Prestini1,
  author = "Elena Prestini",
  title = "A restriction theorem for space curves",
  journal = "Proc. Amer. Math. Soc.",
  volume = "70",
  pages = "8-10",
  year = "1978",
  doi = "10.1090/S0002-9939-1978-0467160-6",
}

@inproceedings{
  Prestini2,
  author = "Elena Prestini",
  title = "Restriction theorems for the Fourier transform to some manifolds in
    \(\mathbb{R}^{n}\)",
  booktitle = "Harmonic Analysis in Euclidean Spaces",
  journal = "Proceedings of Symposia in Pure Mathematics",
  part = "1",
  year = "1979",
  volume = "35.1",
  pages = "101-109",
  doi = "10.1090/pspum/035.1",
}

@inproceedings{
  Ruiz,
  author = "Alberto Ruiz",
  title = "On the restriction of Fourier transforms to curves",
  eventtitle = "Conference on harmonic analysis in honor of Antoni Zygmund",
  year = "1982",
  pages = "186-212",
}

@article{
  Sjolin,
  author = "Per Sjolin",
  title = "Fourier multipliers and estimates of the Fourier transform of
    measures carried by smooth curves in \(\mathbb{R}^{2}\)",
  journal = "Stud. Math.",
  volume = "51",
  pages = "169-182",
  year = "1974",
  doi = "10.4064/sm-51-2-169-182",
}

@article{
  Sogge,
  author = "Christopher D. Sogge",
  title = "A sharp restriction theorem for degenerate curves in
    \(\mathbb{R}^{2}\)",
  journal = "Amer. J. Math.",
  volume = "109",
  number = "2",
  pages = "223-228",
  year = "1987",
  doi = "10.2307/2374572",
}

@article{
  Steinig,
  author = "John Steinig",
  title = "On some rules of Laguerre's, and systems of equal sums of like
    powers",
  journal = "Rend. Math.",
  volume = "6",
  number = "4",
  pages = "629-644",
  year = "1971",
}

@incollection{
  Stein1,
  author = "Elias M. Stein",
  title = "Harmonic analysis on \(\mathbb{R}^{n}\)",
  booktitle = "Studies in harmonic analysis",
  volume = "13",
  pages = "97-135",
  year = "1976",
  publisher = "Mathematical Association of America",
}

@article{
  Stovall,
  author = "Stovall",
  title = "Uniform estimates for Fourier restriction to polynomial curves in \(\mathbb{R}^{d}\)",
  journal = "Amer. J. Math",
  volume = "138",
  number = "2",
  pages = "449-471",
  year = "2016",
  doi = "10.1353/ajm.2016.0021",
}

@online{
  Wan,
  author = "Renhui Wan",
  title = "Restriction estimates for some generalized curves by multilinear decomposition",
  year = "2019",
  eprint = "331357856",
  eprinttype = "ResearchGate",
  urldate = "2022-01-13",
}

@article{
  Zygmund,
  author = "A. Zygmund",
  title = "On Fourier coefficients and transforms of functions of two variables",
  journal = "Studia Mathematica",
  volume = "50",
  number = "2",
  pages = "189-201",
  year = "1974",
  doi = "10.4064/sm-50-2-189-201",
}
\end{document}